\documentclass{amsart}                                       
\usepackage{graphicx}  
\usepackage[usenames, dvipsnames]{color}
\usepackage{hyperref} 
\hypersetup{colorlinks=false, citecolor=green, urlcolor=blue,
         linkcolor=red, breaklinks=true, hypertexnames=true
         }

\newcommand{\ca}{\c c\~a}

\newcommand{\RR}{{\mathbf R}}
\newcommand{\Rn}{\RR^n}
\newcommand{\xy}{(x,y)}

\newcommand{\va}{\varphi}
\newcommand{\mm}{<\!\!<}

\DeclareMathOperator{\Fix}{Fix}
\DeclareMathOperator{\tr}{tr}

\newcounter{lixo}

\newtheorem{theorem}{Theorem}[section]
\newtheorem{lemma}[theorem]{Lemma}
\newtheorem{proposition}[theorem]{Proposition}
\newtheorem{corollary}[theorem]{Corollary}

\newtheorem{definition}[theorem]{Definition}
\theoremstyle{definition}

\title[Synchrony and canards in coupled FitzHugh--Nagumo ]{ Synchrony and canards in \\ two coupled FitzHugh--Nagumo equations\\ \today}  

\author[B.F.F. Gon\c calves]{Bruno F. F. Gon\c calves}
\address{ B.F.F. Gon\c calves --- Centro de Matem\'{a}tica da
   Universidade do Porto   \\
   Rua do Campo Alegre, 687\\
   4169-007 Porto\\
Portugal}

\email{B.F.F. Gon\c calves --- brunoffg9@gmail.com}  

\author[I.S. Labouriau]{Isabel S. Labouriau}
\address{I.S. Labouriau --- Centro de Matem\'{a}tica da
   Universidade do Porto   \\
   Rua do Campo Alegre, 687\\
   4169-007 Porto\\
Portugal}

\email{I.S. Labouriau --- islabour@fc.up.pt}  

\author[A.A. P. Rodrigues]{Alexandre A. P. Rodrigues}
\address{A.A. P. Rodrigues --- Lisbon School of Economics \& Management\\
Centro de Matem\'atica Aplicada \`a\ Previs\~ao e Decis\~ao Econ\'omica\\
 Rua do Quelhas 6\\ 1200-781, Lisboa\\
  Portugal}
\email{A.A. P. Rodrigues --- arodrigues@iseg.ulisboa.pt}

\begin{document}

\begin{abstract}
 We describe the fast-slow dynamics of two FitzHugh--Nagumo equations coupled symmetrically through the slow equations.
   We use symmetry arguments to find a non-empty open set of parameter values for which the two equations synchronise, and another set with antisynchrony --- where the solution of one equation is minus the solution of the other.
   By combining the dynamics within the synchrony  and antisynchrony subspaces, we  also obtain bistability --- where these two types of solution coexist as hyperbolic attractors. They persist under small perturbation of the parameters.
   Canards are shown to give rise to mixed-mode oscillations.
   They also  initiate  small amplitude transient oscillations before the onset of large amplitude relaxation oscillations. We also discuss briefly the effect of asymmetric coupling, with periodic forcing of one of the equations by the other. We illustrate our results with numerical simulations.
\end{abstract}

\maketitle

\begin{center}
   \parbox[t][][t]{.3\linewidth}{
   Bruno F. F. Gon\c calves\\
   {\em corresponding author}\\
   Centro de Matem\'{a}tica\\
      Universidade do Porto   \\
      Rua do Campo Alegre, 687\\
      4169-007 Porto\\
   Portugal}
   \ 
   \parbox[t][][t]{.3\linewidth}{
   Isabel S. Labouriau\\
   Centro de Matem\'{a}tica \\
      Universidade do Porto   \\
      Rua do Campo Alegre, 687\\
      4169-007 Porto\\
   Portugal}
   \ 
   \parbox[t][][t]{.3\linewidth}{
   Alexandre A. P. Rodrigues\\
   Lisbon School of Economics \& Management\\
   Centro de Matem\'atica Aplicada\\
   \`a\ Previs\~ao e Decis\~ao Econ\'omica\\
   Rua do Quelhas 6\\ 1200-781, Lisboa\\
   Portugal}
\end{center}
\bigbreak


 
\textbf{Secondary:} 34C60, 34E17,  34C15, 37G15

\section{Introduction}\label{sec:introduction}
The FitzHugh--Nagumo equations (FHN)  have  been intensively studied in their own right for purely mathematical reasons, as they provide a very simple example of equations that show complex and varied dynamics. 
This is the point of view of the present article.
The equations have been introduced by FitzHugh \cite{Fitzhugh1,Fitzhugh2} as a simplified model for nerve impulse, but we do {\bf not} address this aspect here. 
They have also been presented by Nagumo {\sl et al.} \cite{Nagumo} as a model for experiments in an electrical circuit.

Several different formulations of  the FitzHugh--Nagumo equations\footnote{Nagumo {\sl et al.} \cite{Nagumo} also quote it as Bonhoeffer--van der Pol equations.} \cite{Fitzhugh1,Fitzhugh2,Nagumo} occur in the literature both as reaction-diffusion equations and as ordinary differential equations.
Reviews of the behaviour of a single FHN can be found  for instance in Cebri\'an-Lacasa {\sl et al.} \cite{CPRG2024} focusing on mathematical aspects of the dynamics and, with emphasis on neurobiology,  
in Dmitrichev {\sl et al.} \cite{Dmitrichev}.
Here we see it as the following fast-slow system of ordinary differential equations
\begin{equation}\label{eq:FHN}
   \left\{
   \begin{array}{rcl}
      \varepsilon\dot x&=&4x-x^3-y=f\xy\\
      \dot y&=&x-by-c=g\xy
   \end{array}
   \right.
   \qquad b,c\in \RR\quad 0<\varepsilon\mm 1
\end{equation}  
where $x$ is the fast variable and $y$ is the slow one.
Its dynamics was described in Gon\c calves {\sl et al.} \cite{GLR2024}.
In this article we explore the consequences of coupling two identical FHN.
 
Coupling two FHN has been addressed in many places in the literature. 
For instance, the effect of having the fast variable of a periodic solution of a FHN coupled to the fast equation of another FHN and thus forcing it has been explored numerically by several authors.
It was described as experiments in an electrical circuit by Binczack {\sl et al.} \cite{BinczakEtal2006}, through self-coupling by Desroches {\sl et al.} \cite{DKO2008} and also numerically by Hoff {\sl et al.} \cite{HoffEtAl2014}.
Other authors have used the slow coordinate of a periodic solution of one FHN to force the fast coordinate of another, as in Doss-Bachelet {\sl et al.} \cite{DFP2003} and Krupa {\sl et al.} \cite{Krupa2012}, the latter in a model with three time scales. 
Dmitrichev {\sl et al.} \cite{Dmitrichev} have used the fast coordinate of one FHN to force the slow coordinate of another FHN.
 
Coupling FHN symmetrically through the fast equations has been addressed both numerically and analytically in Pedersen {\sl et al.} \cite{PedersenEtal2022} and Kristiansen and Pedersen \cite{KristiansenPedersen}. 
The synchronisation of FHN with different parameter values coupled in this way was also treated analytically in Plotnikov and Fradkov \cite{Plotnikov}. 
Their dynamics was described numerically by Hoff {\sl et al.} \cite{HoffEtAl2014}.
Both symmetric and asymmetric linear coupling of two FHN through the fast equations was studied analytically by Campbell and Waite \cite{CampbellWaite} and numerically by Santana {\sl et al.} \cite{SantanaEtal2021}.
This type of coupling, but with a delay was studied numerically in Saha and Freudel \cite{Saha2017} and both numerically and experimentally in Ponomarenko {\sl et al.} \cite{Ponomarenko} and in Kulminskiy {\sl et al.} \cite{Kulminskiy} for synchronisation.
Experimental and numerical results on both symmetric and asymmetric sigmoid coupling were described in Egorov {\sl et al.} \cite{Egorov}. 
Two different coupling constants, one for the fast variables the other for the slow ones are explored by Kawato {\sl et al.} \cite{KawatoEtAl1979} and by Krupa {\sl et al.} in 2014 \cite{Krupa2014}.
Synchronisation of two coupled equations of Hodgkin-Huxley type (models for nerve impulse) has been studied by Labouriau and Rodrigues \cite{LabouriauR2003}, but it does not cover FHN.
  
Many of these authors explored the fast-slow structure of the equations focusing on several different aspects of the dynamics: synchrony, periodic and chaotic mixed-mode oscillations, canards, and bursts.
We will discuss their findings in more detail in the final section of this article.
\bigbreak

Here we couple two identical FHN symmetrically and linearly through the slow equations.
To the best of our knowledge this has not been done before, probably because this type of coupling does not have a natural interpretation for nerve cells.
However, this coupling could be realised, for instance, between two Nagumo's circuits, as is suggested in Nagumo {\sl et al.} \cite[Figures 5 and 6]{Nagumo}, though the authors explore the connection in a different way.
It allows us to compare the different types of coupling, as discussed at the end of this article.

\subsection*{Structure of the article}
We explore analytically the two different time scales that arise in \eqref{eq:FHN} for small values of $\varepsilon$ in order to obtain dynamical information on the coupled system.
The fast-slow structure allows us to locate a region in phase space where all the recurring dynamics takes place.
After describing the coupled equations in Section~\ref{sec:FHN} we introduce some of the techniques of fast-slow systems in Section~\ref{sec:critical}, in particular describing an important object in this type of study, the critical manifold.
Conditions under which the solutions of the two equations synchronise are obtained in Section~\ref{sec:synchro}.
A type of global dynamical behaviour called canards that gives rise to mixed-mode oscillations and small amplitude transients, is treated in Section~\ref{sec:MMOs}.
Finally our results are discussed in Section~\ref{sec:discussion} comparing them to the findings by the above mentioned authors, with pointers to future work, followed by a brief excursion into the possibly chaotic mixed-mode oscillations that arise in one-directional coupling where one FHN is used for periodically forcing the other.

We have endeavoured to make a self contained exposition bringing together all related topics. 
We have provided illustrative figures and intuitive descriptions of the mathematical concepts and techniques to make the paper easily readable.
For the reader not familiar with fast-slow systems we suggest reading our paper \cite{GLR2024} where the method is described with the equations \eqref{eq:FHN} as an example.
All figures in this article were created through numerical simulations conducted in \emph{Matlab}, using integration functions such as \emph{ode15s} or \emph{ode23s}, except for non-numerical figures drawn by the authors.

\section{The object of study}\label{sec:FHN}
In this section, we introduce the coupled system under analysis, as well as some terminology and notation that will be used throughout the article.
More precisely, we study the dynamics of two FitzHugh--Nagumo equations coupled linearly through the slow equations:
\begin{equation}\label{eq:2FHN}
   \left\{
   \begin{array}{rcl}
      \varepsilon\dot{x}_1 &=& -y_1+\va(x_1)\\
      \varepsilon\dot{x}_2 &=&-y_2+\va(x_2)\\
      \dot{y}_1 &=&x_1-by_1-c-k(y_1-y_2)\\
      \dot{y}_2 &=& x_2-by_2-c-k(y_2-y_1)	
   \end{array}
   \right.
   \qquad
   \va(x)=4x-x^3
   \qquad b,c,k\in \RR\quad k\ne 0 \quad 0<\varepsilon\mm 1 \ .
\end{equation}
For $i\in \{1,2\}$, $\dot{x_i}, \dot{y_i}$ represent the usual derivative with respect to $t$.
We think of \eqref{eq:2FHN} as two coupled FHN systems, that we will abuse language calling cells, each one with dynamics represented by $(x_1,y_1)$ and $(x_2,y_2)$.
The parameters $b$ and $c$ modulate the dynamics of each cell as discussed in \cite{GLR2024} and in \cite{CPRG2024}.
The parameter $k\ne 0$ is the coupling strength.

When $k=0$, the two cells are independent and the dynamics of \eqref{eq:2FHN} is characterised by the cartesian product of two FitzHugh--Nagumo equations.
As described in \cite{GLR2024}, the dynamics of a single FitzHugh--Nagumo model is characterised by the existence of at least one and at most three equilibrium states, that may be either stable, unstable or saddles.
For some parameter values there is also a type of periodic solution called \emph{relaxation oscillation} whose speed has two times scales.
This implies that for $k=0$ (no coupling) the dynamics of \eqref{eq:2FHN} is characterised by the existence of periodic orbits (cartesian product with one of the equilibria) and by a resonant normally hyperbolic\footnote{See discussion of this concept in Section~\ref{sec:critical} below.} torus foliated by periodic solutions of rotation number 1. 
Hopf bifurcations and \emph{canard explosions} have been found in the particular cases $b=0$ (cf. \cite[Example 6.1]{GLR2024}) and $c=0$ (cf. \cite[Example 6.2]{GLR2024}). 
For $k\neq 0$ the dynamics of each oscillator interferes on the other. 
The normally hyperbolic torus (when it exists) persists for $k\ne 0$ small, as will be discussed in Section~\ref{sec:critical} below, where the concept of {\em normal hyperbolicity} is addressed. 

System \eqref{eq:2FHN} is formulated as a fast-slow system with two fast and two slow equations that are, respectively, the equations for $\dot x_i$ and for $\dot y_i$, $i=1,2$. 
There is a huge literature on this type of system, we refer the reader to the book \cite{Kuehn}, whose treatment we follow. 

A fast-slow system of differential equations is one in which some variables have their derivatives with larger magnitude than others. 
This leads to a system with two time scales. 
The general approach to this type of systems starts by grouping the variables in two disjoint sets: fast variables and slow variables. 
This separation is introduced in system \eqref{eq:2FHN} by the parameter $\varepsilon$. 
Accordingly, we call $X=(x_1,x_2)$ the {\em fast variables} whose dynamics is governed by the {\em fast equations} $\varepsilon\dot X=F(X,Y)$ with  $Y=(y_1,y_2)$ and
$$
F(x_1,x_2,y_1,y_2)=\left(-y_1+\va(x_1),-y_2+\va(x_2)\right)\ .
$$ 
Similarly, the {\em slow variables} $Y=(y_1,y_2)$ obey the {\em slow equations} $\dot Y=G(X,Y)$ with 
$$
G(x_1,x_2,y_1,y_2)=\left(x_1-by_1-c-k(y_1-y_2),x_2-by_2-c-k(y_2-y_1)\right) .
$$

\section{The critical manifold}\label{sec:critical}
An important concept in the dynamics of fast-slow systems is the {\em critical manifold} $C_0$, defined as the set of equilibria of the fast equations (that in general are {\bf not} equilibria of the full system).  
In the present case it is the surface 
$$
C_0=\left\{(x_1,x_2,y_1,y_2)\in \RR^4 :\ y_i=\va(x_i),\ i=1,2\right\}.
$$
In the singular case $\varepsilon=0$ the equations \eqref{eq:2FHN} become a system of two differential equations in the variables $y_i$ subject to the constraint $(x_1,x_2,y_1,y_2)\in C_0$.
The role of $C_0$ in the non singular case $\varepsilon>0$ will be discussed below, but before this we need to introduce another concept.

The critical manifold $C_0$ contains the {\em fold lines} (see Figure~\ref{fig:synchro1}) where its projection into the slow variables $y_1$, $y_2$ is singular.
In the present case this happens when either $\va'(x_1)=0$ or $\va'(x_2)=0$, hence the set of fold lines is given by
\begin{equation}
   \label{def:Sigma}
   \Sigma=\left\{(x_1,x_2,y_1,y_2)\in C_0:\ x_1=\pm 2/\sqrt{3}\quad\mbox{or}\quad x_2=\pm 2/\sqrt{3}
   \right\}.
\end{equation}

The degenerate points $(x_1,x_2,\va(x_1),\va(x_2))\in\Sigma$ with both $x_1=\pm x_2$ and $x_1=\pm 2/\sqrt{3}$ are \emph{double folds}, that correspond to the transverse crossing of two fold lines in $C_0$. 
They are illustrated as red points in Figures~\ref{fig:VarCrit} and \ref{fig:synchro1}.
The sets $C_0$ and $\Sigma$ have the structure of smooth manifolds.

\begin{figure}
   \includegraphics[width=0.7\linewidth]{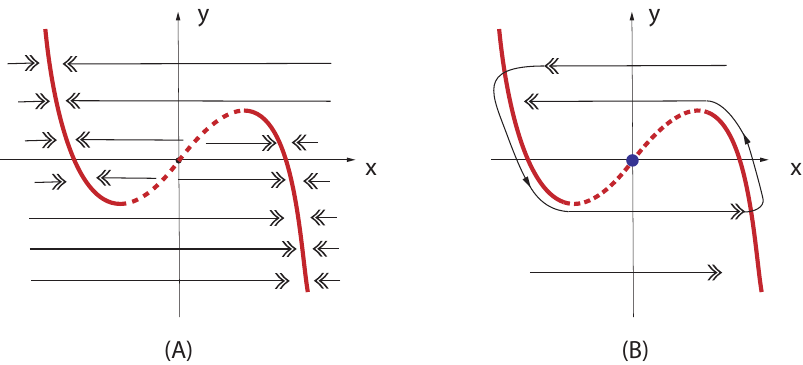}
   \caption{\small Qualitative behaviour of a fast-slow system illustrated for the FitzHugh--Nagumo equations \eqref{eq:FHN} with $b=c=0$ (the van der Pol equations).
            (A) - Solutions move rapidly to the critical manifold $C_0$ (thick red curve),
            then -- as shown in (B) -- they move slowly along the attracting part of  $C_0$ (the solid thick red curve) under the slow equation until they reach a fold point, where they jump out of $C_0$ and move quickly to another attracting part of the critical manifold.
            The dashed curve in (B) is the repelling part of $C_0$ and in a two-dimensional system the folds are isolated points, in this case they are the local extremes of $y=\va(x)$.
         }
   \label{fig:SlowFast}
\end{figure}

An intuitive description of the dynamics in the non singular case $0<\varepsilon\mm 1$ is as follows.
The coordinates $x_1(t)$ and $x_2(t)$ of a solution move a lot faster than the slow coordinates $y_1(t)$ and $y_2(t)$, so we may imagine that the $y_i(t)$ remain constant, while $(x_1(t),x_2(t))$ move fast to an attracting equilibrium of the fast equation, i.e., a point in $C_0$ where $\va'(x_i)<0$, as shown in Figure~\ref{fig:SlowFast} (A) for a single FHN \eqref{eq:FHN}.
Once in the critical manifold, the $x_i$ remain constant and the solutions evolve slowly in $C_0$, governed by the slow equations, and remain in $C_0$ until they reach a point in the fold lines.
At these points one of the $\va'(x_i)$ is zero, so the equilibrium of the fast equation stops being attracting and the solution jumps out of $C_0$ as in Figure~\ref{fig:SlowFast} (B).
Away from $C_0$ the fast equation takes over and the solution moves quickly to another attracting equilibrium of the fast equation.

A rigorous version of the description above was initially obtained in 1952 by Tikhonov \cite{Tikhonov} (quoted, for instance, in \cite{Wasow}) and later extended by Fenichel \cite{Fenichel}. 
Further extensions are described in \cite{Kuehn}.
In order to apply these results we need to define some concepts, valid for any vector field.

\begin{figure}
   \includegraphics[width=0.5\linewidth]{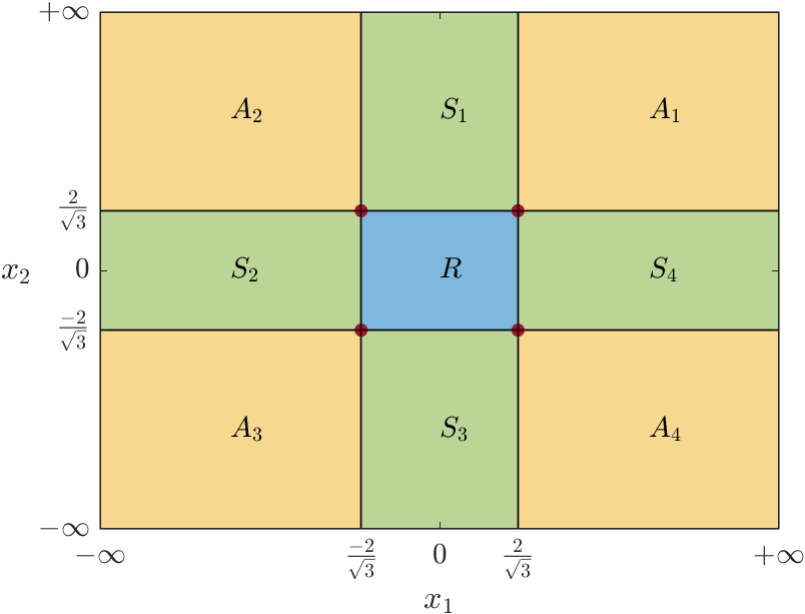}
   \caption{\small The critical manifold $C_0$ projected into the $(x_1,x_2)$ plane as in Lemma~\ref{lema:critManif} with four attracting regions $A_1\cup A_2\cup A_3\cup A_4=A$ (yellow), four saddle regions $S_1\cup S_2\cup S_3\cup S_4=S$ (green) one repelling region $R$ (blue), four fold lines (black) and four double fold points (red dots).
         }
   \label{fig:VarCrit}
\end{figure}

\begin{definition}\label{def:invariant}
   Let $v(\xi)$ be a smooth vector field in $\RR^n$, $S\subset \Rn$ and $M$ a smooth submanifold of $\Rn$. 
   \begin{enumerate}
      \item The set $S$ is {\em locally flow-invariant} under $v$ if for every point $x_0\in S$ there is a $t_0>0$ such that for every $t$ with $|t|<t_0$ the solution $x(t)$ of $\dot x=v(x)$ with initial condition $x(0)=x_0$ lies in $S$.
      \item The set $S$ is {\em flow-invariant} under $v$ if for every point $x_0\in S$ and every $t\in \RR$, the solution $x(t)$ of $\dot x=v(x)$ with initial condition $x(0)=x_0$ lies in $S$.
      \item The manifold $M\subset\RR^n$ is {\em normally hyperbolic} if the component of the vector field $v$ transverse to $M$ is hyperbolic. 
      More precisely, if there is a smooth splitting $\left.T\RR^n\right|_M=TM\oplus N$ such that for all $\xi\in M$ the eigenvalues of $\left.Dv(\xi)\right|_{N}$ have non-zero real part where $T$ represents the tangent space.
   \end{enumerate}
\end{definition}

Note that with these definitions a locally flow-invariant set may not be flow-invariant and a normally hyperbolic manifold is not hyperbolic.
This may be confusing, but these are standard mathematical concepts, in line with the mathematicians' practice of using adverbs.
A result by Fenichel \cite{Fenichel} shows that a manifold that is normally hyperbolic persists under small perturbations of the differential equations.
We will be using this result throughout the article.
In order to explore this feature we establish in the next result that the critical manifold has this property and classify the Lyapunov stability of the equilibria of the fast equations.
Normal hyperbolicity will also be used in the study of synchrony in Section~\ref{sec:synchro} below. 

\begin{lemma}\label{lema:critManif}
   The critical manifold $C_0$ for \eqref{eq:2FHN} is normally hyperbolic with respect to the fast equations everywhere except at the fold points in $\Sigma$.
   The complement $C_0\backslash \Sigma$ has 9 connected components (see Figure~\ref{fig:VarCrit}) corresponding to the different types of stability of equilibria of the fast equations:
   \begin{itemize}
      \item four regions with attracting equilibria, their union is\\
      $A=\left\{(x_1,x_2,y_1,y_2):\  |x_i|>2/\sqrt{3}\quad y_i=\va(x_i)\quad i=1,2\right\}$\\
      where $\partial F/\partial X$ has two  eigenvalues with negative real part;
      \item four regions with saddle equilibria, their union is\\
      $S=\left\{(x_1,x_2,y_1,y_2):\  |x_i|-2/\sqrt{3}\ \mbox{have opposite signs and}\quad y_i=\va(x_i)\quad i=1,2\ \right\}$\\
      where $\partial F/\partial X$ has two  eigenvalues of opposite signs;
      \item one region with  repelling equilibria\\
      $R=\left\{(x_1,x_2,y_1,y_2):\  |x_i|<2/\sqrt{3}\quad y_i=\va(x_i)\quad i=1,2\right\}$\\ where $\partial F/\partial X$ has two eigenvalues with positive real part.
   \end{itemize}
\end{lemma}

The sets $A$ and $S$ of the previous result may be written as the disjoint union described in the caption of Figure \ref{fig:VarCrit}. 
The letters $A$, $S$ and $R$ are suggestive of the stability character of the regions, $A$ for attracting, $S$ for saddle and $R$ for repelling.

\begin{proof}
   At a point $\xi=\left(x_1,x_2,\va(x_1),\va(x_2)\right)\in C_0$ the tangent space $T_\xi C_0$ is generated by $\left(1,0,\va'(x_1),0\right)$ and $\left(0,1,0,\va'(x_2)\right)$.
   If $\xi\notin\Sigma$ then $\va'(x_1)\ne 0$ and $\va'(x_2)\ne 0$, so we may take $N$ to be the plane generated by $\left(1,0,0,0\right)$ and $\left(0,1,0,0\right)$.
   For the vector field $v$ associated to the equations \eqref{eq:2FHN} then $\left.Dv(\xi)\right|_{N}$ is the derivative of the fast equations $F(X,Y)$ with respect to the fast variables $X$.
   Then normal hyperbolicity is determined by the eigenvalues of the matrix 
   $$
   \left.Dv(\xi)\right|_{N}=\dfrac{\partial F}{\partial X}{(x_1, x_2,y_1,y_2)}=
   \begin{pmatrix}\va'(x_1)&0\\0&\va'(x_2)\end{pmatrix}
   \qquad\mbox{where}\qquad \va'(x)=4-3x^2\ .
   $$
   The result follows by inspection of the sign of $\va'$.
\end{proof}

The independence of the time scales when $\varepsilon = 0$ makes the dynamics of \eqref{eq:2FHN} simpler to understand. 
The next result follows directly from Lemma~\ref{lema:critManif} and says that the dynamics of the perturbed system \eqref{eq:2FHN} is similar to that of the singular system, with a deviation of $\mathcal{O}(\varepsilon)$, where $\mathcal{O}$ stands for the usual Landau notation. 
In that result the expression ``close'' refers to the Haussdorf topology.

\begin{corollary}\label{cor:centreManif}
   With respect to \eqref{eq:2FHN}, for every $r\geq 2$ and for sufficiently small $\varepsilon>0$, there is a locally flow-invariant manifold $C_\varepsilon$ of class $C^r$ which is $\mathcal{O}(\varepsilon)$ close to the normally hyperbolic part of $C_0$. 
   Moreover, close to the set $A$ where $C_0$ attracts the fast flow, the manifold $C_\varepsilon$ is also locally attracting for the fast equations in \eqref{eq:2FHN} and there is a locally attracting and locally flow-invariant manifold $A_\varepsilon$.
\end{corollary}

\begin{proof}
   Since the critical manifold $C_0$ is normally hyperbolic everywhere except at the fold points in $\Sigma$, the result follows by Fenichel's Theorem \cite{Fenichel} that guarantees that for $0 < \varepsilon \ll1$, there exists a set $C_\varepsilon$ very close to $C_0$ exhibiting the same dynamical behaviour as $C_0$.
\end{proof}

The manifold $C_\varepsilon$ is called the \emph{slow manifold} associated to \eqref{eq:2FHN}.

For small $\varepsilon$ a typical solution of \eqref{eq:2FHN} behaves as follows: a trajectory starting away from $C_\varepsilon$ moves fast to the attracting part, $A_\varepsilon$, of $C_\varepsilon$ and remains near this sheet, with the dynamics close to that of the slow equations in $C_0$, until it runs into the fold  lines $\Sigma$. 
At $\Sigma$ it will typically jump out of $C_\varepsilon$ and move fast to another attracting component of $C_\varepsilon$ close to $C_0\backslash \Sigma$.
Therefore, for small $\varepsilon>0$, solutions of \eqref{eq:2FHN} will be close to {\em singular solutions}, defined as trajectories that move with the fast equation into the attracting part of $C_0$ and move on $C_0$ following the slow equations $\dot Y=G(X,Y)$.

In special situations a trajectory may cross $\Sigma$ and remain in the unstable region of $C_\varepsilon$ for some time. 
Such a trajectory is called a {\em canard} and will be discussed in Definition \ref{def:canard} below.

\section{Synchrony and bistability} \label{sec:synchro}
In this section we use symmetries to obtain conditions on the parameters of \eqref{eq:2FHN} under which its solutions are in {\em synchrony}: they behave like a single FitzHugh--Nagumo system.
We show that the set of such synchronous solutions contains a subset that attracts all solutions that start near it.
We also obtain conditions for {\em antisynchrony} where one of the ``cells'' in  \eqref{eq:2FHN} is exactly minus the other, that also contains an attracting subset.
We show that  the two situations are compatible, giving rise to bistability in the equations, and we obtain conditions where these properties hold in an approximate sense when the two coupled equations are not identical.
The proofs will involve normal hyperbolicity and rely on Fenichel's Theorem \cite{Fenichel}.
We start by a rigorous description of the concepts. 

\begin{definition}\label{def:synchro}
   Let $\xi(t)=\left(x_1(t),x_2(t),y_1(t),y_2(t)\right)$, $t\in \RR$, be a solution of \eqref{eq:2FHN}. 
   Then we define:
   \begin{itemize}
      \item $\xi(t)$ is in {\em synchrony} if we have $x_1(t)=x_2(t)$ and $y_1(t)=y_2(t)$ for all $t\in\RR$;
      accordingly the plane ${\mathcal S}=\left\{(x,x,y,y):\quad x,y\in\RR\right\}\subset\RR^4$ is called the {\em synchrony plane};
      \item $\xi(t)$ is in {\em antisynchrony} if we have $x_1(t)=-x_2(t)$ and $y_1(t)=-y_2(t)$ for all $t\in\RR$; accordingly the plane ${\mathcal AS}=\left\{(x,-x,y,-y):\quad x,y\in\RR\right\}\subset\RR^4$ is called the {\em antisynchrony plane}.
   \end{itemize}
\end{definition}

We start by addressing synchrony.
\begin{theorem}\label{teo:synchro}
   The synchrony plane ${\mathcal S}$ is flow-invariant under \eqref{eq:2FHN}.

   For small $\varepsilon> 0$ the intersection ${\mathcal S}\cap A_\varepsilon$ of the synchrony plane with the attracting part, $A_\varepsilon$, of the slow manifold $C_\varepsilon$ is locally flow-invariant under \eqref{eq:2FHN}. 
   It is normally hyperbolic and locally attracting if $k> -b/2$.
\end{theorem}

\begin{proof}
   Equations \eqref{eq:2FHN} have the symmetry $\gamma(x_1,x_2,y_1,y_2)=(x_2,x_1,y_2,y_1)$, therefore the synchrony plane ${\mathcal S}$ is the set of fixed points of this symmetry ${\mathcal S}=\Fix(\gamma)$ and hence it is flow-invariant.  
   The rest of the proof is divided in two lemmas stated and proved below.

   In Lemma~\ref{prop:synchro} we obtain conditions for the intersection of ${\mathcal S}=\Fix(\gamma)$ with the attracting part, $A$, of $C_0$ to be attracting. 
   Then, in Lemma~\ref{cor:synchro} we deal with the intersection with $A_\varepsilon$.
\end{proof}

\begin{figure}
  \parbox{0.35\linewidth}{
  \begin{center}
      \includegraphics[width=\linewidth]{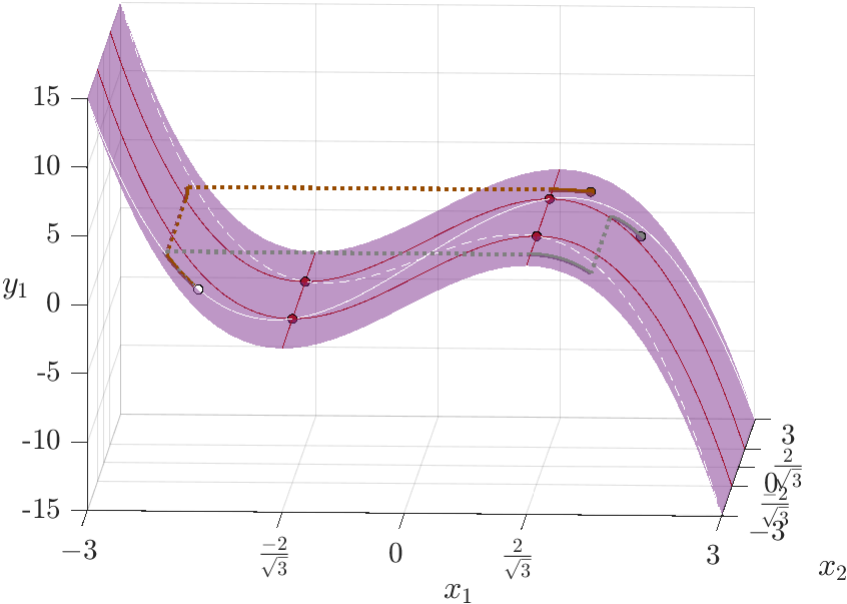}\\
      (A)            
   \end{center}
   }
   \quad
   \parbox{0.35\linewidth}{
   \begin{center}
      \includegraphics[width=\linewidth]{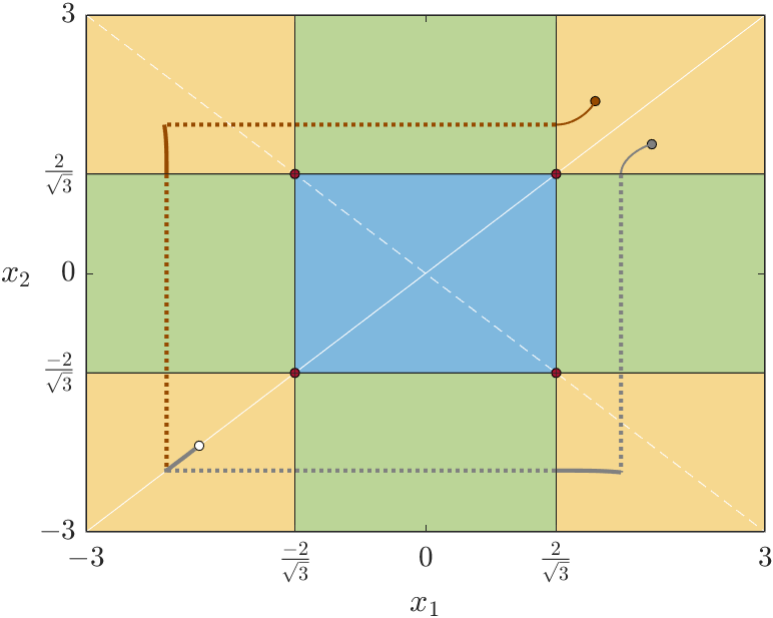}\\
      (B)            
   \end{center}
   } 
   \caption{\small Illustration of Lemma~\ref{prop:synchro}. 
            Two singular solutions of \eqref{eq:2FHN} attracted to the synchrony plane shown in two different projections, parameters $b=0$, $c=-2$ and $k=1$.
            Conventions: trajectories with initial conditions at the brown and gray dots, slow part in brown/gray solid lines, fast part dotted, stable equilibrium (white dot), fold lines (red/black), double fold points (red dots), intersection with synchrony plane white line. 
            (A) - Projection of $C_0$ into the subspace $\left(x_1,x_2,y_1\right)$ (purple) and trajectories. 
            (B) - Same trajectories as in (A) projected into the plane $\left(x_1,x_2\right)$. 
         }
   \label{fig:synchro1}
\end{figure}
        
\begin{figure}
  \parbox{0.35\linewidth}{
  \begin{center}
      \includegraphics[width=\linewidth]{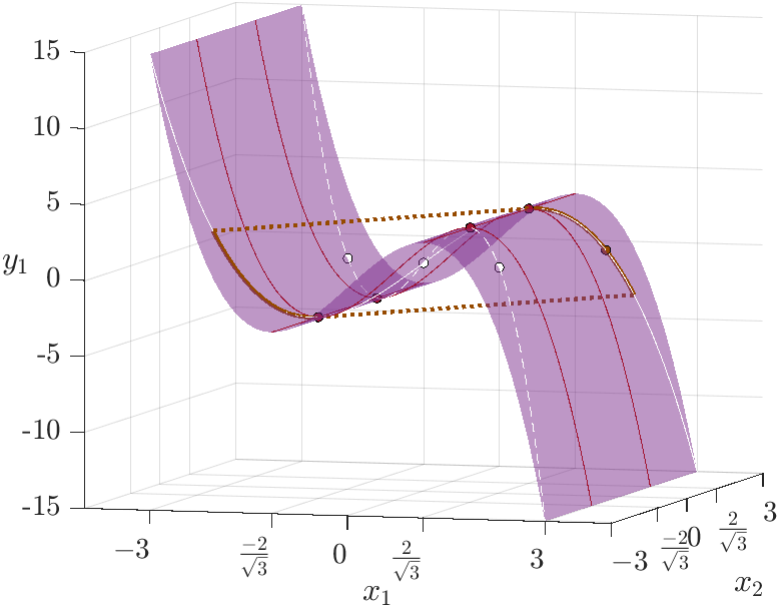}\\
      (A)            
   \end{center}
   }
   \quad
   \parbox{0.35\linewidth}{
   \begin{center}
      \includegraphics[width=\linewidth]{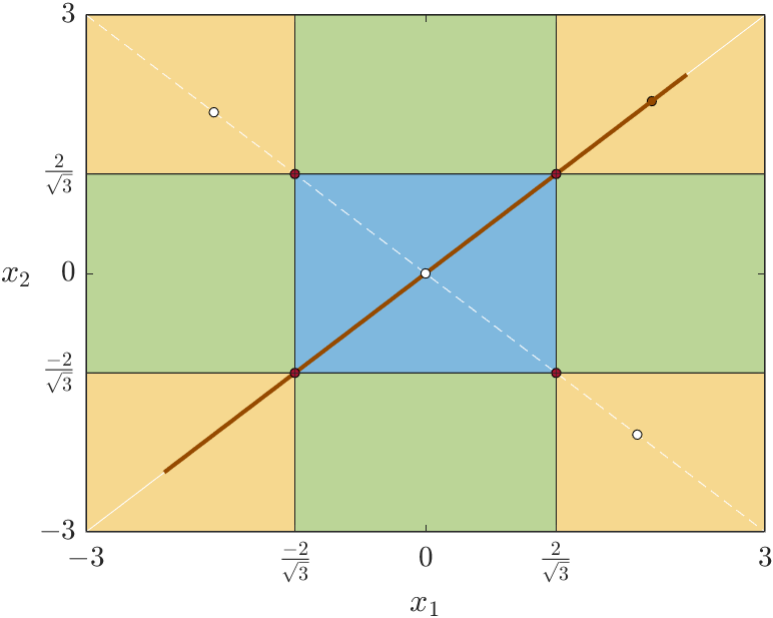}\\
      (B)            
   \end{center}
   }
   \caption{\small Illustration of Theorem~\ref{teo:synchro}. 
            Singular solution of \eqref{eq:2FHN} on the synchrony plane shown in two different projections, parameters  $b=0=c$ and $k=1$.
            Conventions: trajectories with initial conditions at the brown dot, slow part in brown solid lines, fast part dotted, equilibria (white dots), fold lines (red/black), degenerate fold points (red dots), intersection with synchrony plane (white line). 
            (A) - Projection $C_0$ into the subspace $\left(x_1,x_2,y_1\right)$ (purple). 
            (B) - Same trajectories as in (A) projected into the plane $\left(x_1,x_2\right)$.  
         }
   \label{fig:synchro3}
\end{figure}

In the next result we use properties of linear maps on $\RR^4$ to analyse the stability. 
The expressions $\tr$ and $\det$ denote the usual trace and determinant operators. 

\begin{lemma}\label{prop:synchro}
   The intersection $\Fix(\gamma)\cap A$ is
   normally hyperbolic if and only if $k\ge -b/2$.
   Moreover, in this case it is locally attracting.
\end{lemma}

\begin{proof}
   The normal hyperbolicity of the invariant plane $\Fix(\gamma)$ is determined by the eigenvalues of the component $N$ of the derivative of the vector field $V(X,Y)=(F(X,Y),G(X,Y))$ transverse to $\Fix(\gamma)$. 
   We change coordinates to $z_1=x_1+x_2$, $z_2=y_1+y_2$ in $\Fix(\gamma)$ and $z_3=x_1-x_2$, $z_4=y_1-y_2$ in $\Fix(\gamma)^\perp$, the orthogonal complement of $\Fix(\gamma)$.
   This means that:
   $$
   x_1= \frac{z_1+z_3}{2}, \quad x_2= \frac{z_1-z_3}{2}, \quad y_1= \frac{z_2+z_4}{2}\quad \text{and}\quad y_1= \frac{z_2-z_4}{2}.
   $$
   Let $\widehat{V}(Z)$ be the expression of the vector field $V(X,Y)$ in the new coordinates $Z=(z_1,z_2,z_3,z_4)$, with associated equations given by
   \begin{equation}\label{eq:changed}
      \left\{\begin{array}{lcl}
         \varepsilon \dot{z}_1&=&-z_2 + \varphi\left(\dfrac{z_1+z_3}{2}\right)+\varphi\left(\dfrac{z_1-z_3}{2}\right)\\
         \dot{z}_2&=& z_1 -bz_2 -2c\\
         \varepsilon \dot{z}_3&=&  -z_4 + \varphi\left(\dfrac{z_1+z_3}{2}\right)-\varphi\left(\dfrac{z_1-z_3}{2}\right)\\
         \dot{z}_4&=& z_3 -bz_4 -2k z_4 .
      \end{array}\right.
   \end{equation}
   
   Computing $D\widehat{V}(Z)$ at the plane $\Fix(\gamma)$, given by $x_1=x_2=x$, $y_1=y_2=y$ where $z_1=2x$, $z_2=2y$, $z_3=z_4=0$ we obtain:
   \begin{equation}\label{eq:DVchanged}
      \left.D{ \widehat{V}}(z_1,z_2,z_3,z_4)\right|_{(2x,2y,0,0)}=\begin{pmatrix}
         \va'(x) /\varepsilon&-1/\varepsilon&0&0\\
         1&-b&0&0\\
         0&0&\va'(x)/\varepsilon&-1/\varepsilon\\
         0&0&1&-(b+2k)
         \end{pmatrix}
         \qquad
         N_\gamma(x)=\begin{pmatrix}\va'(x)/\varepsilon&-1 /\varepsilon\\1&-(b+2k)
      \end{pmatrix}
   \end{equation}
   where $N_\gamma(x)$ is the component of $D\widehat{V}(Z)$ in the directions perpendicular to $\Fix(\gamma)$, evaluated at $Z=(2x,2y,0,0)\in \Fix(\gamma)$, with
   $$
   \det(N_\gamma(x))=- \dfrac{1}{\varepsilon}\left( (b+2k)\va'(x)-1\right)
   \quad \text{and}\quad \tr(N_\gamma(x))=\dfrac{\va'(x)}{\varepsilon}-(b+2k).
   $$
   The synchrony plane is normally hyperbolic at $x\in\Fix(\gamma)$ if $\det(N_\gamma(x))\ne 0$ and $\tr(N_\gamma(x))\ne 0$.
   It is attracting at $x\in\Fix(\gamma)$ if $\det(N_\gamma(x))>0$ and $\tr(N_\gamma(x))<0$.
   
   As we saw in Section~\ref{sec:critical}, points in the critical manifold attract the fast flow if $\va'(x_i)<0$ for $i=1,2$.
   So, if $k\ge -b/2$ then $\Fix(\gamma)\cap A$ is normally hyperbolic and attracting, since $\det (N_\gamma(x)) >0$ and $\tr(N_\gamma(x))<0$.

   Finally, to see that the condition $k\ge -b/2$ is necessary for $\Fix(\gamma)$ to be normally hyperbolic, note that in $\Fix(\gamma)\cap A$ the derivative $\va'(x)=4-3x^2$ takes all values in the interval $\left(-\infty,0\right)$.
   Therefore, if $b+2k<0$ there will be some point in $\Fix(\gamma)\cap A_\varepsilon$ where $\va'(x)=1/(b+2k)$, hence $\det(N_\gamma(x))=0$. 
   At this point $0$ is an eigenvalue of $N_\gamma$, so normal hyperbolicity fails.
   Moreover, there will also be some points where either $\tr(N_\gamma(x))>0$, or where $\det(N_\gamma(x))<0$, so some trajectories contained in this set may be attracting, but not all of them.
\end{proof}

\begin{figure}
   \parbox{0.35\linewidth}{
   \begin{center}
      \includegraphics[width=\linewidth]{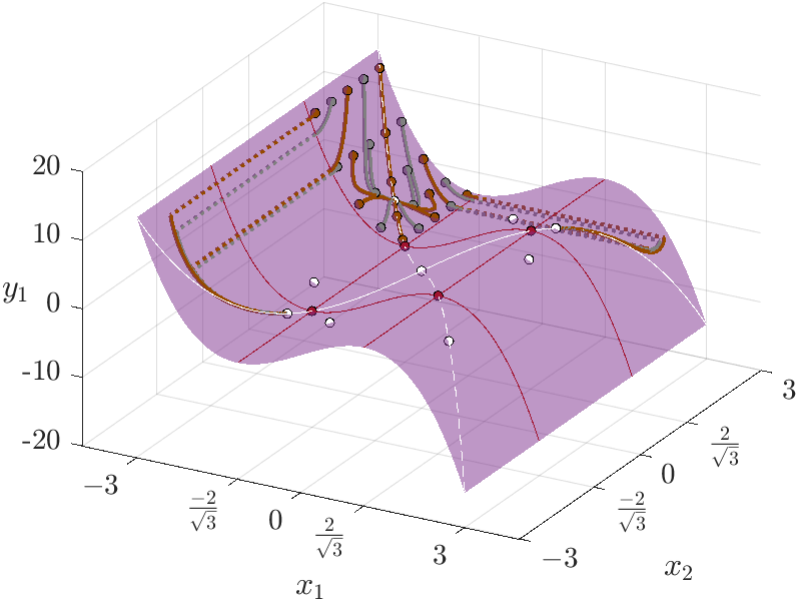}\\
      (A)            
   \end{center}
   }
   \quad
   \parbox{0.35\linewidth}{
   \begin{center}
      \includegraphics[width=\linewidth]{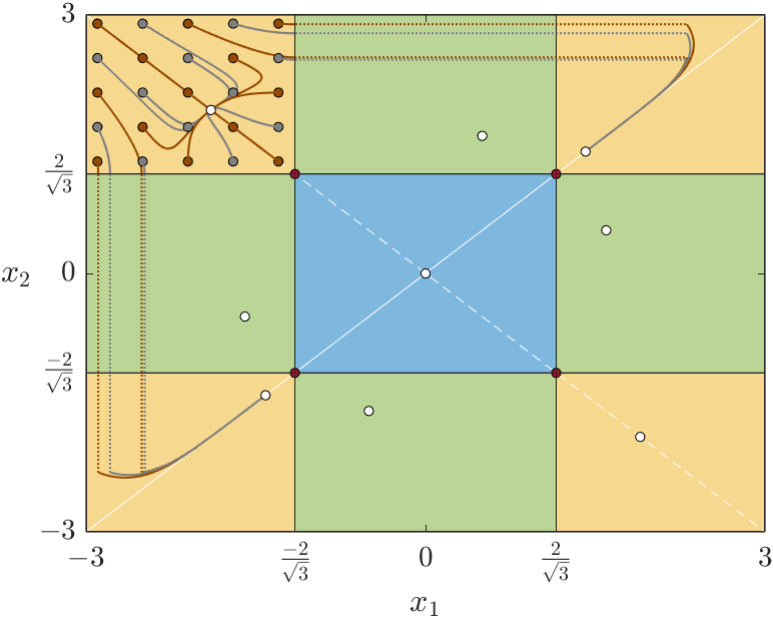}\\
      (B)            
   \end{center}
   }
   \caption{\small  Illustration of Theorems~\ref{teo:synchro} and \ref{teo:antisynchro}. 
            Several singular solutions of \eqref{eq:2FHN} starting in region $A_2$ shown in two different projections, parameters $b=0.5$, $c=0$ and $k=1$.
            Trajectories starting close to the fold line are attracted to the synchrony plane; other trajectories go to $\Fix(\delta)$.
            Conventions: trajectories with initial conditions at the brown and gray dots, slow part in brown/gray solid lines, fast part dotted lines, equilibria white dots, fold lines (red/black), double fold points red dots, intersection with synchrony/anti-synchrony planes (solid/dotted lines), respectively.  
            (A) - Projection of $C_0$ into the subspace $\left(x_1,x_2,y_1\right)$ (purple) and trajectories. 
            (B) - Same trajectories as in (A) projected into the plane $\left(x_1,x_2\right)$.
         }
   \label{fig:synchro2}
\end{figure}
        
\begin{figure}
   \parbox{0.35\linewidth}{
   \begin{center}
      \includegraphics[width=\linewidth]{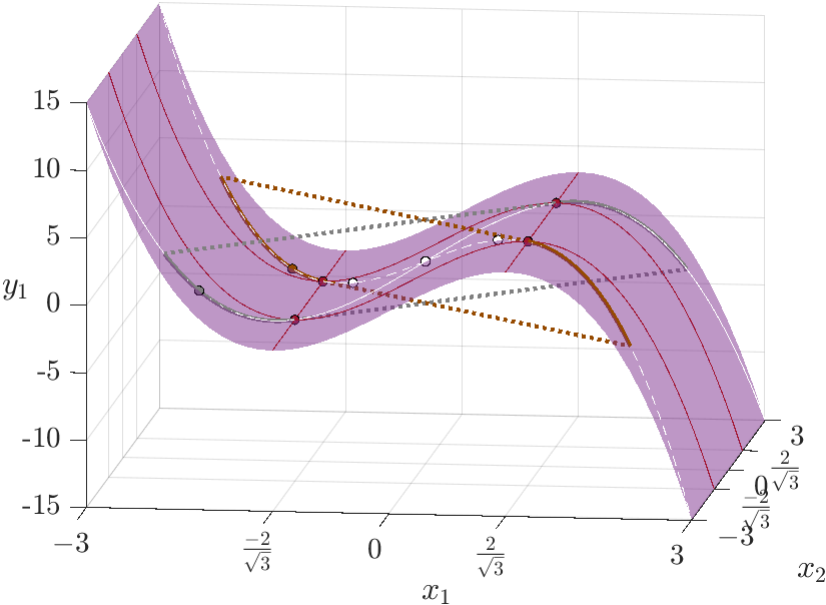}\\
      (A)            
   \end{center}
   }
   \quad
   \parbox{0.35\linewidth}{
   \begin{center}
      \includegraphics[width=\linewidth]{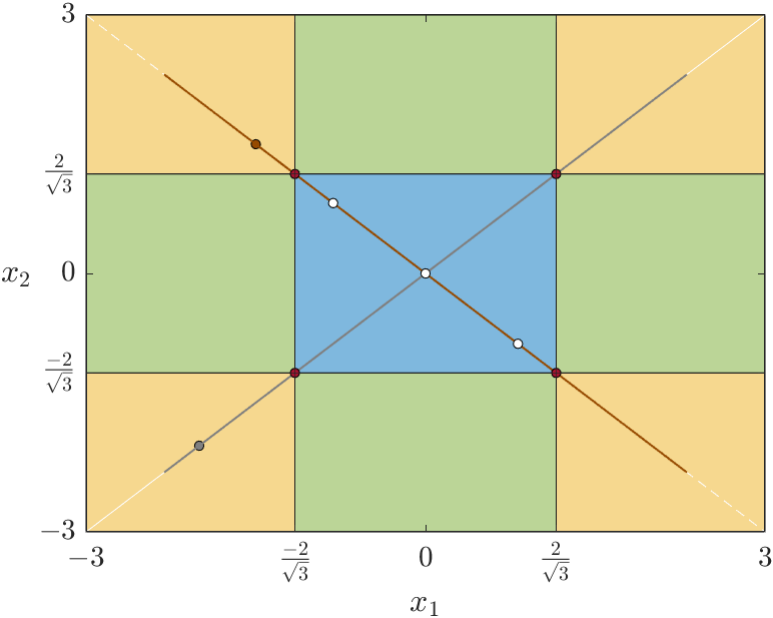}\\
      (B)            
   \end{center}
   }
   \caption{\small Illustration of Theorems~\ref{teo:synchro} and \ref{cor:synchro}. 
            Two singular solutions of \eqref{eq:2FHN} starting in  $\Fix(\gamma)\cap A_2$ and in $\Fix(\delta)\cap A_3$ shown in two different projections, parameters $b=0.1$, $c=0$ and $k=0.1$. 
            Two distinct stable solutions - bistability.
            Conventions: trajectories with initial conditions at the brown and gray dots, slow part in brown/gray solid lines, fast part dotted lines, equilibria white dots, fold lines (red/black), double fold points red dots, intersection with synchrony/anti-synchrony planes (solid/dotted lines), respectively.
            (A) - Projection of $C_0$ into the subspace $\left(x_1,x_2,y_1\right)$ (purple) and trajectories. 
            (B) - Same trajectories as in (A) projected into the plane $\left(x_1,x_2\right)$.
         }
   \label{fig:nova}
\end{figure}        

It remains to see whether the intersection of the synchrony plane with the attracting part, $A_\varepsilon$, of $C_\varepsilon$ is also attracting.  

\begin{lemma}\label{cor:synchro}
  For small $\varepsilon> 0$ the intersection $\Fix(\gamma)\cap A_\varepsilon$ is locally flow-invariant under \eqref{eq:2FHN}. 
  It is normally hyperbolic and locally attracting if $k>-b/2$.
\end{lemma}

\begin{proof} 
   Since $\gamma$ is a symmetry of \eqref{eq:2FHN} then the synchrony plane $\Fix(\gamma)$ is flow-invariant.
   By Corollary~\ref{cor:centreManif} the attracting part $A_\varepsilon$ of the manifold $C_\varepsilon$ is locally flow-invariant.
   Since both $\Fix(\gamma)$ and $C_\varepsilon$ are locally flow-invariant sets, their intersection has the same property.

   In the proof of Lemma~\ref{prop:synchro} it is shown that if $k\ge -b/2$ then, at any point $Z\in\Fix(\gamma)\cap A$, the component $N_\gamma(x)$ of $D\hat{V}(Z)$ perpendicular to $\Fix(\gamma)$ satisfies $\det \left(N_\gamma(x)\right)>0$ and $\tr \left(N_\gamma(x)\right)<0$.
   From Corollary~\ref{cor:centreManif} we know that $A_\varepsilon$ is ${\mathcal O}(\varepsilon)$ close to the synchrony plane $A$, where $\va'(x)<0$.
   Therefore, since $k\ge -b/2$, for small enough $\varepsilon$ and if $\tilde{Z}=(\tilde x,\tilde x,\tilde y,\tilde y)\in\Fix(\gamma)\cap A_\varepsilon$, then for sufficiently small $\varepsilon>0$ we have $\va(\tilde x)<1/(b+2k)$. 
   Therefore $\det \left(N_\gamma(\tilde x)\right)>0$ and $\tr \left(N_\gamma(\tilde x)\right)<0$ and the second statement follows.
\end{proof}

Equilibria in the synchrony plane satisfy both $y=\va(x)=4x-x^3$ and $x-b\va(x)-c=0$.
Therefore in $\Fix(\gamma)$ there is always at least one equilibrium and there are at most three equilibria --- details in Section 3 of \cite{GLR2024}.
By the B\'ezout Theorem there are at most 9 equilibria of \eqref{eq:2FHN} so there may be up to 3 symmetry related pairs of equilibria outside the synchrony plane.
The next result is an example of the dynamics in the synchrony plane. 

\begin{corollary}\label{cor:Relax}
   For small $\varepsilon>0$ if $b>1/4$ and $k>-1/8$ then for small values of $c$ the equations \eqref{eq:2FHN} have a synchronous asymptotically stable relaxation oscillation.
\end{corollary}

\begin{proof}
   The result follows from the fact that under these conditions on the parameters $b$ and $c$ the single FHN \eqref{eq:FHN} has an asymptotically stable relaxation oscillation as shown in \cite{GLR2024}.
   This follows by the Poincar\'e-Bendixson Theorem, because for small values of the parameter $c$ and $b>1/4$ there is only one equilibrium of \eqref{eq:FHN} and it lies on the repelling part of the critical manifold. 

   The dynamics in the synchrony plane is that of a single FHN, hence \eqref{eq:2FHN} has a synchronous relaxation oscillation that, in the restriction to the synchrony plane, is asymptotically stable. 
   If $k>-1/8$ then Theorem~\ref{teo:synchro} holds, so the intersection of the synchrony plane with $A_\varepsilon$ is attracting, therefore the relaxation oscillation is asymptotically stable. 
\end{proof}

Theorem~\ref{teo:synchro} is illustrated by singular solutions shown in Figures~\ref{fig:synchro1},~\ref{fig:synchro3}, \ref{fig:synchro2} and \ref{fig:nova}.
In the example of Figure~\ref{fig:synchro3} we have $b=c=0$, so this is the interaction of two van der Pol equations (see, for instance \cite[Section~1.3]{Kuehn}).
Figures~\ref{fig:synchro2} and \ref{fig:nova} illustrate the meaning of locally attracting: the synchrony plane attracts trajectories in an open set around it, but not all trajectories. 
In the examples of the figures we have $c=0$, hence \eqref{eq:2FHN} has additional symmetry and this is explored in the next result.

\begin{theorem}\label{teo:antisynchro}
   If $c=0$ then the antisynchrony plane ${\mathcal AS}$ is flow-invariant under \eqref{eq:2FHN}. 
   For small $\varepsilon> 0$ if $c=0$ then the intersection ${\mathcal AS}\cap A_\varepsilon$ of the antisynchrony plane with the attracting part of the slow manifold $C_\varepsilon$, $A_\varepsilon$, is locally flow-invariant under \eqref{eq:2FHN}. 
   It is normally hyperbolic and locally attracting if $b>0$.
\end{theorem}

\begin{proof}
   If $c=0$, then the equations \eqref{eq:2FHN} have the additional symmetry 
   $$
   \delta(x_1,x_2,y_1,y_2)=(-x_2,-x_1,-y_2,-y_1).
   $$
   As in the proof of Theorem~\ref{teo:synchro} the plane $\{(x,-x,y,-y)\quad x,y\in\RR\}=\Fix(\delta)={\mathcal AS}$ of points fixed by $\delta$ is flow-invariant.
   It remains to check that it is normally hyperbolic and to see where it is attracting.
   In the coordinates $z_j$ used in the proof of Lemma~\ref{prop:synchro} the subspace $\Fix(\delta)$ is defined by the equalities $z_1=z_2=0$, $z_3=2x$ and $z_4=2y$.
   The expressions for the two matrices $\left.D{ \widehat{V}}(z_1,z_2,z_3,z_4)\right|_{(2x,2y,0,0)}$ and $\left.D{  \widehat{V}}(z_1,z_2,z_3,z_4)\right|_{(0,0,2x,2y)}$ are the same.
   Hence, the component $N_\delta(x)$ of $D{\widehat{V}}$ in the directions perpendicular to $\Fix(\delta)$ is given by:
   $$
   N_\delta(x)=\begin{pmatrix}\va'(x)/\varepsilon&-1/\varepsilon\\1&-b
   \end{pmatrix}
   $$
   and then $$\det(N_\delta(x))=-\frac{1}{\varepsilon} (b \va'(x)-1) \quad \text{and}\quad \tr(N_\delta(x))=\frac{\va'(x)}{\varepsilon}-b.$$
   Using the arguments of the proof of Lemma~\ref{prop:synchro} it follows that $\Fix(\delta)\cap A$ is normally hyperbolic if and only if $b\ge 0$ and in this case it is locally attracting.
   The last statement follows from the same arguments as Lemma~\ref{cor:synchro}.
\end{proof}

It follows from Theorems~\ref{teo:synchro} and \ref{teo:antisynchro} that if $c=0$ and both $b>0$ and $k>-b/2$ there is \emph{bistability} i.e., the system has two stable coexisting states. 
In particular, if $0<b<1/4$ and $c=0$ then the uncoupled FHN in equations \eqref{eq:FHN} have an asymptotically stable periodic solution, as established in \cite{GLR2024} for the proof of Corollary~\ref{cor:Relax}. 
In this case, provided $k> -b/2$, equations~\eqref{eq:2FHN} have two stable periodic solutions each one lying in one of the fixed-point subspaces for the two symmetries. 
Moreover, the bistability persists as shown in the next result. 

\begin{corollary}\label{cor:bistabilityPersists}
   If $b> 0$ then for sufficiently small $c\ne 0$ and for every $r>0$ there is a $C^r$ locally flow-invariant locally attracting manifold for \eqref{eq:2FHN} close to $\Fix(\delta)\cap A_\varepsilon$.
   If moreover $k>-b/2$, then there is bistability in the sense that this manifold and the locally invariant manifold $\Fix(\gamma)\cap A_\varepsilon$ are simultaneously locally attracting.
\end{corollary}

\begin{proof}
   The result follows from the normal hyperbolicity established in Theorems~\ref{teo:synchro} and \ref{teo:antisynchro} and from the persistence of normally hyperbolic locally invariant manifolds proved by Fenichel in \cite{Fenichel71} (see also \cite[Section 2.2]{Kuehn}).
\end{proof}

Similarly, from the persistence of normally hyperbolic invariant manifolds we obtain persistence of the dynamics with respect to the parameters. 
For this we need a definition.

\begin{definition}\label{def:appsynchro}
   Let $\xi(t)=\left(x_1(t),x_2(t),y_1(t),y_2(t)\right)$ be a solution of a system of two coupled equations
   $$
   \left\{\begin{array}{rcl}
   \varepsilon\dot x_i&=&f_i(x_1,x_2,y_1,y_2)\\
   \dot y_i&=&g_i(x_1,x_2,y_1,y_2)
   \end{array}\right.
   \qquad i=1,2 .
   $$
   Then we define:
   \begin{itemize}
      \item $\xi(t)$ is in {\em approximate synchrony} with precision $\Delta>0$ if there is a $t_0\in\RR$ such that for all $t>t_0$ we have $\left|x_1(t)-x_2(t)\right|<\Delta$ and $\left|y_1(t)-y_2(t)\right|<\Delta$;
      \item $\xi(t)$ is in {\em approximate antisynchrony} with precision $\Delta>0$ if there is a $t_0\in\RR$ such that for all $t>t_0$ we have $\left|x_1(t)+x_2(t)\right|<\Delta$ and $\left|y_1(t)+y_2(t)\right|<\Delta$.
   \end{itemize}
\end{definition}

For the original system \eqref{eq:2FHN} we obtain from Corollary~\ref{cor:bistabilityPersists} a result analogous to Corollary~\ref{cor:Relax}.

\begin{corollary}\label{cor:RelaxAnti}
   For small $\varepsilon>0$ if $b>1/4$ then for small values of $c$ the equations \eqref{eq:2FHN} have an asymptotically stable relaxation oscillation in approximate antisynchrony with precision $\Delta={\mathcal O}(c)$.
   If moreover $k>-1/8$ then the attracting relaxation oscillations in synchrony and antisynchrony (in ${\mathcal S}$ and ${\mathcal AS}$) coexist.
\end{corollary}

If the two coupled FHN have slightly different parameter values, then synchrony and antisynchrony still hod in the approximate sense above, as in the next result.

\begin{corollary}\label{cor:synchroPersists}
   Consider a perturbation of \eqref{eq:2FHN} where the equation for $\dot y_2$ is replaced by
   $$
   \dot{y}_2 =x_2-\tilde by_2-\tilde c-\tilde k(y_2-y_1)\ .            
   $$
   Then if $k>-b/2$ and $\tilde k> -\tilde b/2$, for small $\varepsilon>0$ and for sufficiently small $|b-\tilde b|$, $|c-\tilde c|$, $|k-\tilde k|$ the perturbed system has stable solutions with approximate synchrony with precision $\Delta={\mathcal O}\left(|b-\tilde b|, |c-\tilde c|,|k-\tilde k|\right)$.

   If moreover $b> 0$ and $\tilde b>0$ then, for sufficiently small $c$ and $\tilde c$, the perturbed system has stable solutions with approximate antisynchrony with precision $\Delta={\mathcal O}\left(|b-\tilde b|, |c-\tilde c|,|k-\tilde k|\right)$, and hence it has bistability.
\end{corollary}

\begin{proof}
   From the normal hyperbolicity established in Theorems~\ref{teo:synchro} and \ref{teo:antisynchro} and from Fenichel's theorem \cite{Fenichel71} it follows that for sufficiently small $|b-\tilde b|$, $|c-\tilde c|$ and $|k-\tilde k|$ if $k> -b/2$ and $\tilde k> -\tilde b/2$ there is an attracting normally hyperbolic locally invariant manifold for the perturbed system close to $\Fix(\gamma)\cap A_\varepsilon$.
   Also if $b> 0$ and $\tilde b> 0$ then, for sufficiently small $c$ and $\tilde c$, the perturbed system has an attracting normally hyperbolic locally invariant manifold close to $\Fix(\delta)\cap A_\varepsilon$.
\end{proof}

\section{Canards and mixed-mode oscillations}\label{sec:MMOs}
In Section~\ref{sec:FHN} we stated that, in general, solutions of a fast-slow system evolve close to the slow manifold under the slow equations and jump out of it when they reach the fold points, where the slow manifold becomes unstable. 
However, as remarked in Section~\ref{sec:critical} there are special situations when trajectories called {\em canards} may follow the unstable part of the slow manifold for a considerable amount of time, as drawn in Figure~\ref{fig:Canard2d}.
In geometric terms a canard solution corresponds to the intersection of an attracting and a repelling slow manifold near a non-hyperbolic point of $\Sigma$, the set of fold lines defined in \eqref{def:Sigma}.
In this section, we establish some hypotheses ensuring the existence of canards for \eqref{eq:2FHN} and explore some of the consequences for the dynamics.
A description of the dynamics of mixed-mode oscillations associated to canards in general can be found in Desroches {\sl et al.} \cite{Desroches2012} and for the forced van der Pol equations in Guckenheimer \cite{Guckenheimer2025}.

\begin{figure}
   \includegraphics[width=0.25\linewidth]{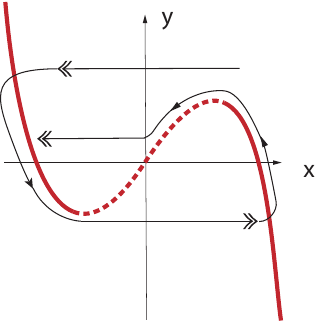}
   \caption{\small Schematic representation of a {\em canard}: a trajectory that follows the unstable part of the slow manifold for some time, before jumping out of it.
            Here the critical manifold is the thick red curve, solid in the attracting part and dashed in the repelling part.
         }
   \label{fig:Canard2d}
\end{figure}

\begin{definition}
   \label{def:canard}
   A trajectory that, after starting $\mathcal{O}(\varepsilon)$ close to the attracting region of the slow manifold, remains $\mathcal{O}(\varepsilon)$ close to the non-attracting region of the slow manifold for a time of order $\mathcal{O}(1)$ is called a {\em canard}.
\end{definition}

Intuitively, we are looking at a trajectory that evolves near the attracting part of the critical manifold.
Then it reaches a fold point, where the critical manifold changes from attracting to repelling with respect to the fast equations.
A canard happens if around this point the fast equations not only have an equilibrium but  are specially slow.
In this way the slow equation dominates the dynamics and the trajectory continues to follow close to the critical manifold for some time, near its unstable part, like in Figure~\ref{fig:Canard2d}.
In a more rigorous setting, Szmolyan and Wechselberger \cite{SW01}, and Krupa {\sl et al.} \cite{Krupa2014} have established that, for one-dimensional and two-dimensional slow equations, canards appear around some points where, after a rescaling that we will describe below, the slow equations have an equilibrium at a fold point.

\subsection{An open set of parameters where canards occur}\label{sec:foldedEq}
In order to find the canards we start by rewriting the slow equations in terms of the fast variables by differentiating implicitly the condition $y_i=\va(x_i)$, $i\in \{1,2\}$ that defines the critical manifold, to yield
\begin{equation}\label{eq:slowx1}
   \va'(x_i)\dot x_i= \dot y_i=x_i-(b+k)\va(x_i)+k\va(x_j)-c
   \qquad i,j\in\{1,2\}, \quad i\ne j.
\end{equation}

A point $X^*= (x_1^*,x_2^*,y_1^*,y_2^*)$ lies on the fold lines $\Sigma$ if and only if either $\va'(x_1^*)=0$ or $\va'(x_2^*)=0$.
Without loss of generality we take $\va'(x_1^*)\ne 0$ and $\va'(x_2^*)=0$, hence $x_2^*= 2\sigma/\sqrt{3}$, $\sigma=\pm 1$, the case $\va'(x_1^*)= 0$ and $\va'(x_2^*)\ne 0$ being identical, due to the symmetry.
The case $\va'(x_1^*)=0=\va'(x_2^*)$ is treated separately in Subsection \ref{subsec:doublefolds} below.

Since $\va'(x_2^*)=0$, the equation obtained from $\dot y_2$ in \eqref{eq:slowx1} yields no dynamical information at $x_2=x_2^*$.
We overcome this by a time rescaling of $\tau=t/\va'(x_2)$, that is singular at $x_2^*$ since $\va'(x_2^*)=0$, as mentioned above. 
Writing $dx_i/d\tau=x_i'$ the equations \eqref{eq:slowx1} transform  into:
\begin{equation}\label{eq:slowx2}
   \left\{\begin{array}{lclcl}
   x_1'&=&\dfrac{\va'(x_2)}{\va'(x_1)}\left(
   x_1-(b+k)\va(x_1)+k\va(x_2)-c\right)&=&H_1(x_1,x_2)\\
   x_2'&=&x_2-(b+k)\va(x_2)+k\va(x_1)-c&=&H_2(x_1,x_2) \ .
   \end{array}\right.
\end{equation}
Let $H(x_1,x_2)=\left(H_1(x_1,x_2),H_2(x_1,x_2)\right)$ for $(x_1, x_2)\in\RR^2\backslash\{(x_1, x_2): \varphi'(x_1)=0\}$ be the vector field corresponding to \eqref{eq:slowx2}.

If the fold point $X^*$ is also an equilibrium of \eqref{eq:slowx2} then the fast equation is slow enough around it to yield a canard.
However, we need to make sure that $X^*$ lies at the boundary of the of the attracting part $A$ of the critical manifold $C_0$.
Given the choice $x_2^*= 2\sigma/\sqrt{3}$ this imposes conditions on $x_1^*$, that we deal with in the next result.

\begin{lemma}\label{lema:foldEq}
   For every $k\ne 0$ and for any choice of the parameters $b$ and $c$ there is an equilibrium $ (x_1^*,x_2^*)$ of \eqref{eq:slowx2} with $x_2^*= 2\sigma/\sqrt{3}$, $\sigma=\pm 1$ and $\left| x_1^*\right|>\dfrac{2}{\sqrt{3}}$.
   Therefore the point $X^*= (x_1^*,x_2^*,\va(x_1^*), \va(x_2^*))$ lies at the boundary of the attracting part $A$ of  the critical manifold $C_0$.
\end{lemma}

\begin{proof}
   Taking into account \eqref{eq:slowx1}, at $x_2^*= 2\sigma/\sqrt{3}$ we have $\va'(x_2^*)=0$ so $x_1'=0$.
   Therefore $\left(x_1^*,2\sigma/\sqrt{3}\right)$, $\sigma=\pm 1$, is an equilibrium of \eqref{eq:slowx2} if and only if $x_2'=0$ i.e., when
   \begin{equation}\label{eq:x1*}
      k\va (x_1^*)=c+(b+k)\va\left(2\sigma/\sqrt{3}\right)-2\sigma/\sqrt{3}
   \end{equation}
   and the first part of the result follows immediately since for $|x|>2/\sqrt{3}$ the function $\va(x)$ takes all the values in $\RR$. The second part is immediate taking into account that $\left| x_1^*\right|>2/\sqrt{3}$.
\end{proof}

Note that an equilibrium $(x_1^*,x_2^*)$ of \eqref{eq:slowx2} may {\em not} correspond to an equilibrium $X^*= \left(x_1^*,x_2^*,\va(x_1^*), \va(x_2^*)\right)$ of the original equations \eqref{eq:2FHN}.
Lemma~\ref{lema:foldEq} shows that there is a non-empty open set of parameters $(k,b,c)$ of \eqref{eq:2FHN} such that there is an   equilibrium $ (x_1^*,x_2^*)$ of \eqref{eq:slowx2} for which the point $ \left(x_1^*,x_2^*,\va(x_1^*), \va(x_2^*)\right)$ lies at the boundary of the attracting set $A$.
  
The time rescaling that we have used allows us to gain enough hyperbolicity to obtain a complete analysis by standard methods from dynamical systems theory.
We call the point $ (x_1^*,x_2^*)$ obtained in Lemma~\ref{lema:foldEq} a {\em folded equilibrium} of \eqref{eq:slowx2}.
For the sake of completeness, we are going to review a way to detect canards of \eqref{eq:2FHN} from folded equilibria, critical points where the slow manifold  generically changes the stability. 
See also Guckenheimer \& Haiduc \cite{GuckenheimerHaduc} and Guckenheimer \cite{Guckenheimer}.

\begin{definition}[Definition 2.1 of \cite{SW01}, adapted]
   Solutions of the reduced problem \eqref{eq:slowx2} passing through a fold point from an attracting region to a repelling one are called singular canards.  
\end{definition}
 
In order to have a canard it is not enough to have a folded equilibrium---it is also necessary to have a trajectory that approaches it in positive time. 
This is the idea behind the next result.

\begin{lemma}[Lemma 2.3 of \cite{SW01}, adapted]
   \label{lemma_canards}
   A canard of \eqref{eq:2FHN} corresponds to, at least, one real negative eigenvalue of the matrix $DH$ at a folded equilibrium.
\end{lemma}

The proof of Lemma \ref{lemma_canards} uses blow-up of singularities and foliation by center manifolds as main techniques.
The blowing up makes the fine structure of the singular perturbation problem visible and then the full problem can be analysed by standard methods from dynamical systems theory.
We return now to our original problem.
 
\begin{lemma}\label{lema:dH}
   For $k\ne 0$, a folded equilibrium $(x_1^*,2\sigma/\sqrt{3})$, with $\sigma=\pm 1$, of the equation \eqref{eq:slowx2} is either a saddle or an unstable node or an unstable focus or a saddle-node. There exists a non-empty open set of the parameters $(k,b,c)\in \RR^3$ for which it is either a saddle or an unstable node.
\end{lemma}
 
\begin{proof} 
   The derivative $DH(x_1,x_2)$ at a folded equilibrium $(x_1^*,x_2^*)$, $x_2^*= 2\sigma/\sqrt{3}$, $\sigma=\pm 1$ is given by
   \begin{equation}\label{eq:DH}
      DH(x_1^*,x_2^*)=\begin{pmatrix}0&
      \dfrac{\va''(x_2^*)}{\va'(x_1^*)}\left(x_1^*-(b+k)\va(x_1^*)+k\va(x_2^*)-c\right)\\
      k\va'(x_1^*)&1
      \end{pmatrix} \ .
   \end{equation}
   Therefore $(x_1^*,x_2^*)$ is either a saddle or an unstable node or an unstable focus or a saddle-node for \eqref{eq:slowx2} since at least one of the eigenvalues is either positive or has positive real part because $\tr DH(x_1^*,x_2^*)=1$.

   The point $(x_1^*,x_2^*)$ is either a saddle or  an unstable node if and only if the eigenvalues of $ DH(x_1^*,x_2^*)$ are real and not zero, i.e., when $0\ne \det DH(x_1^*,x_2^*)\le 1/4$.
   From \eqref{eq:DH}, we have:
   $$
   \det DH(x_1^*,x_2^*)=-k\va''(x_2^*)\left(x_1^*-(b+k)\va(x_1^*)+k\va(x_2^*)-c\right).
   $$
   Replacing $\va(x_1^*)$ by the value in \eqref{eq:x1*} the previous equality implies
   \begin{equation}\label{eq:detDH}
      \det DH(x_1^*,x_2^*)=-k\va''(x_2^*)\left[x_1^*+x_2^*-b\left(\va(x_1^*)+\va(x_2^*)\right)-2c\right]
   \end{equation}
   which is not zero and less than 1/4 for an open set of parameters $(k,b,c)$, since $\va''(x_2^*)=-6x_2^*\neq 0$ and $k\neq 0$.
\end{proof}
 
\begin{figure}
   \parbox{0.3\linewidth}{\begin{center}
		\includegraphics[width=\linewidth]{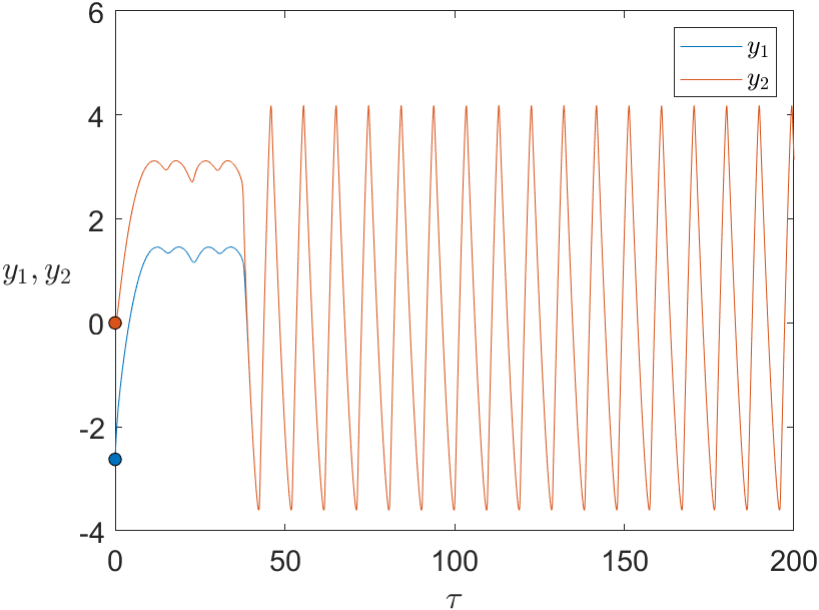}
		\\ (A)
	\end{center}}
	   \parbox{0.3\linewidth}{\begin{center}
		\includegraphics[width=\linewidth]{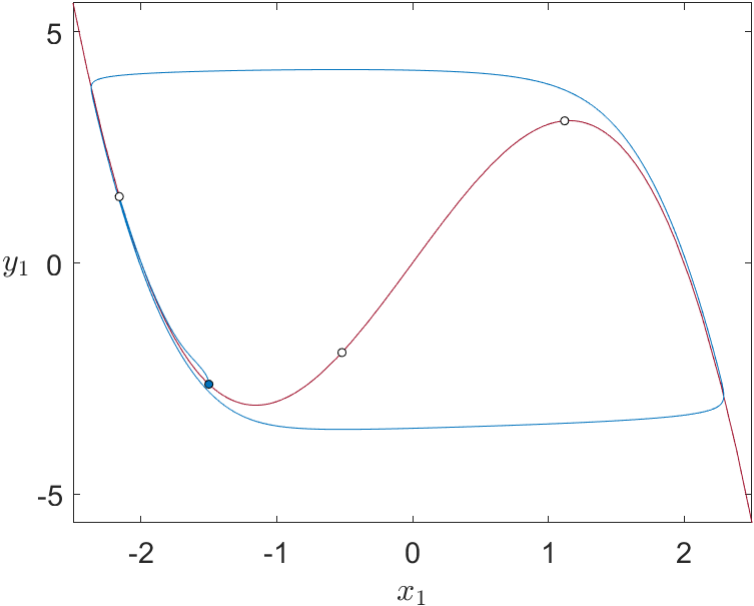}
		\\ (B)
	\end{center}}
   \parbox{0.3\linewidth}{\begin{center}
    	\includegraphics[width=\linewidth]{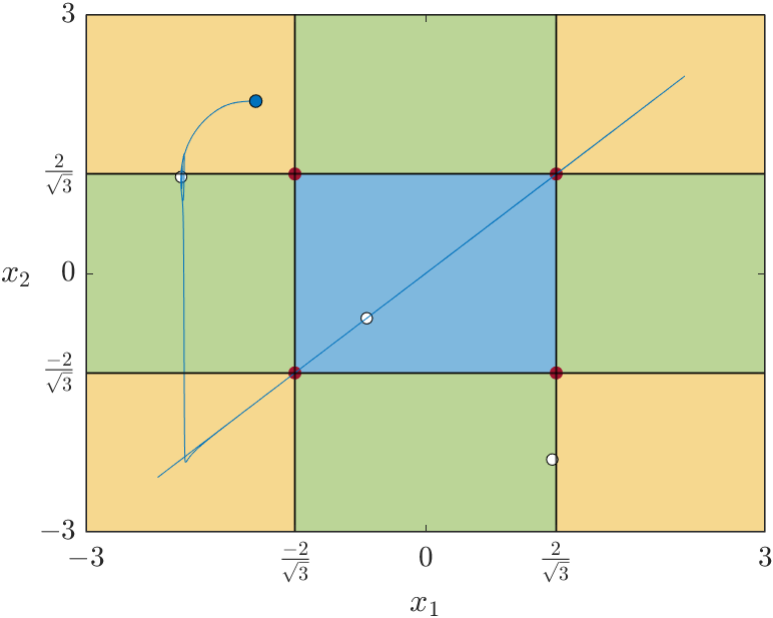}
    	\\ (C)
   \end{center}}
   \\
   \parbox{0.3\linewidth}{\begin{center}
		\includegraphics[width=\linewidth]{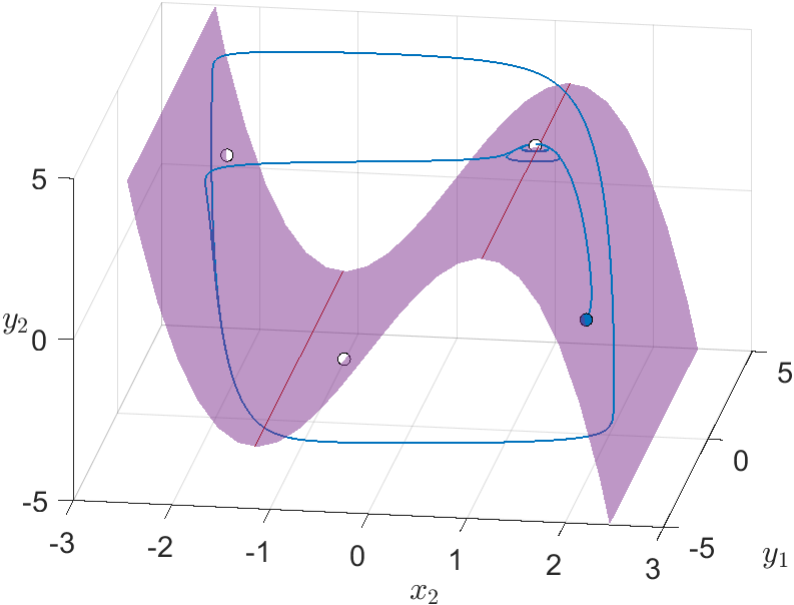}
		\\ (D)
	\end{center}}
   \parbox{0.3\linewidth}{\begin{center}
		\includegraphics[width=\linewidth]{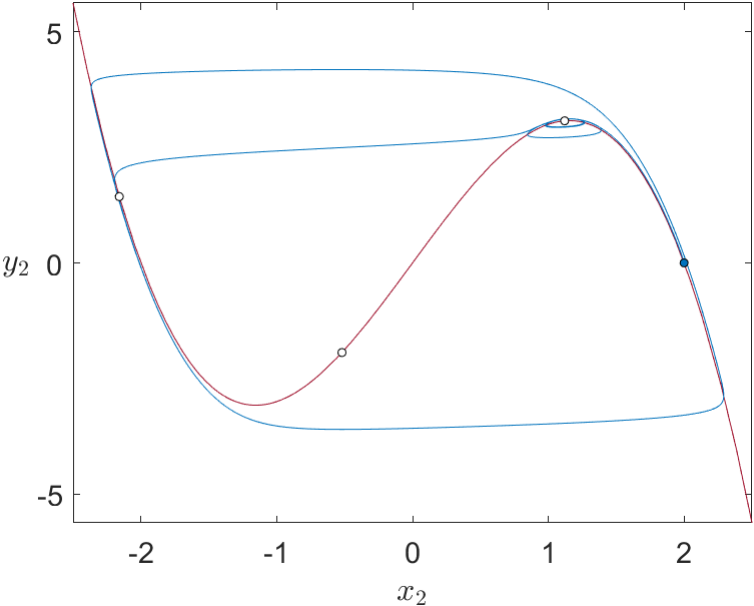}
		\\ (E)
	\end{center}}
   \parbox{0.3\linewidth}{\begin{center}
		\includegraphics[width=\linewidth]{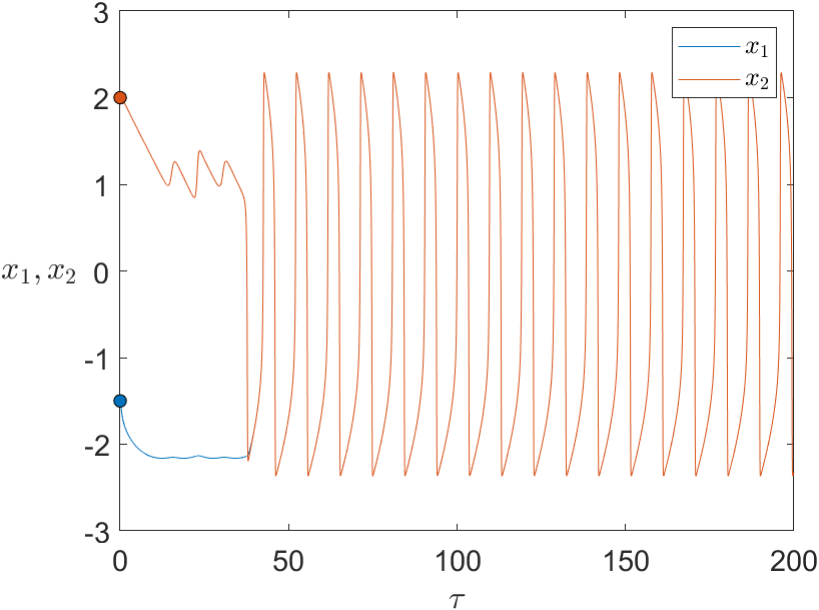}
		\\ (F)
	\end{center}}
	\caption{\small Illustration of Proposition~\ref{cor:opensetcanard}. 
            Canard transient on a solution of \eqref{eq:2FHN} near a folded node, parameters $b=0$, $c=-0.519935054$, $k=1$ and $\varepsilon=0.5$, initial condition $\left(x_{1},x_{2},y_{1},y_{2}\right)=\left(-1.5,2,{ \va(-1.5),\va(2)}\right)$. 
            (A) - Time course for $y_1(t)$ (blue) and $y_2(t)$ (orange). 
            (B) - Projection of the trajectory (blue) on the $\left(x_1,y_1\right)$ plane, initial condition on the blue dot, red critical manifold $C_0$. 
            (C) - Projection of the trajectory (red) on the $\left(x_1,x_2\right)$ plane, initial condition on the red dot, equilibria on the  white dots. 
            (D) - Projection of the trajectory (purple) on the $\left(x_2,y_1,y_2\right)$ space, initial condition on the blue dot, purple critical manifold $C_0$. 
            (E) - Projection of the trajectory (blue) on the $\left(x_2,y_2\right)$ plane, conventions as in (B) -- transient small oscillations near folded node. 
            (F) - Time course for $x_1(t)$ (blue) and $x_2(t)$ (orange).
         }
	\label{fig:canard}
\end{figure}
 
The next result is the goal of this subsection:
\begin{proposition}\label{cor:opensetcanard}
   If $\left| x_1^*\right|>\dfrac{2}{\sqrt{3}}$ and  $x_2^*= 2\sigma/\sqrt{3}$, $\sigma=\pm 1$, there exists a non-empty open set of the parameters $(k,b,c)\in \RR^3$ for which  there is at least one trajectory of  \eqref{eq:2FHN} that  goes across $X^*=\left(x_1^*,x_2^*,\va(x_1^*), \va(x_2^*)\right)$  generating a canard.
\end{proposition}

\begin{proof}
   First of all note that the time rescaling $\tau=t/\va'(x_2)$ we have used reverses time orientation when $\left(x_1^*,x_2^*,\va(x_1^*), \va(x_2^*)\right)$ is in the attracting part of $C_0$ (note that here we have $\va'(x_2^*)<0$).  
   Therefore when $(x_1^*,x_2^*)$ is either a saddle or an unstable node for \eqref{eq:slowx2} there is at least one trajectory of \eqref{eq:2FHN} that goes across $\left(x_1^*,x_2^*,\va(x_1^*), \va(x_2^*)\right)$ generating a canard.
   This is a consequence of Lemma \ref{lemma_canards}, where it is stated that canards correspond to, at least, one negative eigenvalue of $DH$ associated to \eqref{eq:slowx2} at a fold point.   
\end{proof}

Examples are shown in Figures~\ref{fig:canard}, \ref{fig:MMO1} and \ref{fig:MMO2}. In the first of these examples, before the trajectory approaches a synchronous periodic orbit it makes some transient small oscillations around the fold point (``canard behaviour-type'') that are visible in Figure~\ref{fig:canard}~(E). In this example there is a folded node, see Table~\ref{table:folded} below. 
 
Note that from equation \eqref{eq:detDH} it follows that for any given $k\ne 0$ and $b\in \RR$ there is a value of the parameter $c\in \RR$ for which $\det DH(x_1^*,x_2^*)=0$.
This corresponds to saddle-node bifurcations.
According to \cite{Krupa2014} this gives rise to \emph{mixed-mode oscillations} like those shown in Figures~\ref{fig:MMO1} and \ref{fig:MMO2}, i.e., trajectories that combine small oscillations and large oscillations of relaxation type, both recurring in an alternating manner.

\subsection{The case $b=0$}\label{subsec:b0}
As an illustration of Lemmas~\ref{lema:foldEq} and  \ref{lema:dH} we describe their contents in a particular case.
More information on the equilibria of the equation \eqref{eq:slowx2} may be obtained if we make the simplifying assumption $b=0$.
Indeed the expression of \eqref{eq:FHN} in the particular case $b=0$ has been used in the analysis of two FHN coupled in the fast equations by Pedersen {\sl et al.}  \cite{PedersenEtal2022}, in experimental results with an electrical circuit by Kulminskiy {\sl et al.}  \cite{Kulminskiy} and with double coupling fast--to--fast and slow--to--slow by Krupa {\sl et al.} 2014  \cite{Krupa2014}.
Thus this particular case is interesting for comparing the outcomes of different types of couplings.
We obtain sufficient conditions on the parameters $c$ and $k\neq 0$ for the existence of a folded saddle.
Since by Lemma~\ref{lema:dH} this holds in an open set of parameters, then the folded saddles will persist for small $b\ne 0$.

From expression \eqref{eq:x1*} the folded equilibrium $(x_1^*,x_2^*)$ of \eqref{eq:slowx2} with $x_2^*=2\sigma/\sqrt{3}$, $\sigma=\pm1$, satisfies the equality
\begin{equation}\label{eq:dva}
   k\left(\va(x_1^*)-\va(x_2^*)\right)=c-x_2^*.
\end{equation}
Substituting into \eqref{eq:detDH} we get
\begin{equation}\label{eq:detDHb0}
    \det DH(x_1^*,x_2^*)= -k\va''(x_2^*)\left(x_1^*-k\va(x_1^*)+k\va(x_2^*)-c\right)=
    -k\va''(x_2^*)\left[x_1^*+x_2^*-2c\right].   
\end{equation}

The folded equilibria $(x_1^*,2\sigma/\sqrt{3})$ may then be classified, in some cases under additional conditions on the sign of $\va(x_2^*)-\va(x_1^*)$, as shown in Table~\ref{table:folded}.
To do this we divide the region $A$ in four components $A=A_1\cup A_2\cup A_3 \cup A_4$, as depicted in Figure \ref{fig:VarCrit}, where:
\begin{eqnarray*}
   A_1&=&\left\{(x_1,x_2,y_1,y_2) \in\RR^4:\  x_1>2/\sqrt{3}, \quad x_2>2/\sqrt{3}, \quad y_i=\va(x_i),\quad i=1,2\right\}\\
   A_2&=&\left\{(x_1,x_2,y_1,y_2) \in\RR^4:\  x_1<-2/\sqrt{3}, \quad x_2>2/\sqrt{3}, \quad y_i=\va(x_i),\quad i=1,2\right\}\\
   A_3&=&\left\{(x_1,x_2,y_1,y_2) \in\RR^4:\  x_1<-2/\sqrt{3}, \quad x_2<-2/\sqrt{3}, \quad y_i=\va(x_i),\quad i=1,2\right\}\\
   A_4&=&\left\{(x_1,x_2,y_1,y_2) \in\RR^4:\  x_1>2/\sqrt{3}, \quad x_2<-2/\sqrt{3}, \quad y_i=\va(x_i),\quad i=1,2\right\}.
\end{eqnarray*}
We will use the notation $A_i^*$ for the sets $A_i^*=\left\{(x_1,x_2):\ \left(x_1,x_2,\va(x_1),\va(x_2)\right)\in A_i\right\}$, $i=1,\ldots,4$.

In Table~\ref{table:folded}, we present sufficient conditions on $c$ for the existence and classification of a folded equilibrium of \eqref{eq:slowx2} with $b=0$ on the components with $x_2^*= \pm 2/\sqrt{3}$ of $\partial A_i^*$, $i=1,\ldots,4$, the boundary of the attracting region of the critical manifold. 
The analysis for folded equilibria with $x_1^*= \pm 2/\sqrt{3}$ follows by symmetry.
We present the computations for the cases where the folded equilibrium $\left(x_1^*,x_2^*\right)$ lies on the boundaries of $A_1^*$ or $A_2^*$.
The computations for the other cases in Table~\ref{table:folded} run along the same lines. For $(x_1^*,x_2^*)\in \partial A_1^*$ or $\partial A_3^*$ the classification is simpler, since $\va(x_1^*)-\va(x_2^*)$ has constant sign.

\subsection*{Case $\partial A_1^*$}  
If the folded equilibrium $\left(x_1^*,x_2^*\right)=\left(x_1^*,2/\sqrt{3}\right)$ lies on the boundary, $\partial A_1^*$, of region $A_1^*$ then the following conditions hold:
$$
\varphi''(x_2^*)=-6x_2^*<0
\qquad
x_1^*-x_2^*>0
\qquad
x_1^*+x_2^*>4/\sqrt{3}>0
\qquad
\varphi(x_1^*)-\varphi(x_2^*)<0.
$$
We start by obtaining sufficient conditions on the parameter $c$ for the existence of a folded equilibrium, and then proceed to classify it:
\begin{enumerate}
   \item  If $c<2/\sqrt{3}=x_2^*$ then \eqref{eq:dva} implies that $k>0$.
   Therefore $\det H(x_1^*, x_2^*)=-k\va''(x_2^*)\left[x_1^*+x_2^*-2c\right] >0$, and hence $\left(x_1^*,x_2^*\right)$ is either a node or an unstable focus of \eqref{eq:slowx2}.
 
   \item If $c>2/\sqrt{3}=x_2^*$, then \eqref{eq:dva} implies that $k<0$.
   In this case there are the following possibilities:
   \begin{enumerate}
      \item If $c>\left(x_1^*+2/\sqrt{3}\right)/2$ then $\det H(x_1^*, x_2^*)>0$ and  $\left(x_1^*,x_2^*\right)$ is either a node or an unstable focus of \eqref{eq:slowx2}. 
      \item If $c<\left(x_1^*+2/\sqrt{3}\right)/2$ then $\det H(x_1^*, x_2^*)>0$ and  $\left(x_1^*,x_2^*\right)$ is a saddle.
   \end{enumerate}
\end{enumerate}

\subsection*{Case $\partial A_2^*$} 
Similarly to what we did before, if the folded equilibrium $\left(x_1^*,x_2^*\right)=\left(x_1^*,2/\sqrt{3}\right)$ lies on the boundary, $\partial A_2^*$, of region $A_2^*$ then we have
$$
\varphi''(x_2^*)=-6x_2^*<0
\qquad
x_1^*<-2/\sqrt{3}=-x_2^*.
$$
We divide the calculation in two main cases:
\begin{enumerate}
   \item If $\va(x_1^*)-\va(x_2^*)>0$, then there are three sufficient conditions on the parameter $c$ for the existence of a folded equilibrium of \eqref{eq:slowx2}. 
   \begin{enumerate}
      \item If $c>2/\sqrt{3}$ then $c-x_2^*>0$ and $x_1^*+x_2^*-2c<0$. 
      Therefore $k>0$ (by \ref{eq:dva}) and hence $\det DH(x_1^*,x_2^*)<0$. Then $(x_1^*,x_2^*)$ is a saddle of \eqref{eq:slowx2}.
      \item If $0<c<2/\sqrt{3}$ then $c-x_2^*<0$ and $x_1^*+x_2^*-2c<0$.
      Therefore $k<0$, $\det DH(x_1^*,x_2^*)>0$  and $(x_1^*,x_2^*)$ is either a node or an unstable focus of \eqref{eq:slowx2}.
      \item If $x_1^*-c>0$ then  $c<0$ and $x_1^*+x_2^*-2c>0$.
      From \eqref{eq:dva} it follows that $k<0$ implying $\det DH(x_1^*,x_2^*)>0$ and $(x_1^*,x_2^*)$ is either a node or an unstable focus of \eqref{eq:slowx2}.
   \end{enumerate}
   \item If $\va(x_1^*)-\va(x_2^*)<0$ then the conditions above on $c$ are applicable, but $\det DH(x_1^*,x_2^*)$ has the opposite sign as discussed above.
\end{enumerate}

\begin{table}[h]
   \begin{center}
      \begin{tabular}{ccclccc}
         region	&	$x_1^*$	&	$x_2^*$	&	$c$	&	other 	&	$k$	&	equilibrium	\\
            &		&		&		&	condition	&		&	type	\\ \hline
         $\partial A_1^*$	&	$x_1^*>2/\sqrt{3}$	&	$+2/\sqrt{3}$	&	$c<2/\sqrt{3}$	&	-	&	$k>0$	&	node or focus	\\
            &		&		&	$2/\sqrt{3}<c<\left(x_1^*+x_2^*\right)/2$	&	-	&	$k<0$	&	saddle	\\
            &		&		&	$c>\left(x_1^*+x_2^*\right)/2$	&		&	$k<0$	&	node or focus	\\ \hline
         $\partial A_2^*$	&	$x_1^*<-2/\sqrt{3}$	&	$+2/\sqrt{3}$	&	$c>2/\sqrt{3}$	&	$\va(x_1^*)>\va(x_2^*)$	&	$k>0$	&	saddle	\\
            &		&		&	$0<c<2/\sqrt{3}$	&	$\va(x_1^*)>\va(x_2^*)$	&	$k<0$	&	node or focus	\\
            &		&		&	$c<x_1^*$	&	$\va(x_1^*)>\va(x_2^*)$	&	$k<0$	&	node or focus	\\
            &		&		&	$c>2/\sqrt{3}$	&	$\va(x_1^*)<\va(x_2^*)$	&	$k<0$	&	node or focus	\\
            &		&		&	$0<c<2/\sqrt{3}$	&	$\va(x_1^*)<\va(x_2^*)$	&	$k>0$	&	saddle	\\
            &		&		&	$c<x_1^*$	&	$\va(x_1^*)<\va(x_2^*)$	&	$k>0$	&	saddle	\\ \hline
         $\partial A_3^*$	&	$x_1^*<-2/\sqrt{3}$	&	$-2/\sqrt{3}$	&	$c>-2/\sqrt{3}$	&	-	&	$k>0$	&	node or focus	\\
            &		&		&	$\left(x_1^*+x_2^*\right)/2<c<-2/\sqrt{3}$	&	-	&	$k<0$	&	saddle	\\
            &		&		&	$c<\left(x_1^*+x_2^*\right)/2$	&		&	$k<0$	&	node or focus	\\ \hline
         $\partial A_4^*$	&	$x_1^*>2/\sqrt{3}$	&	$-2/\sqrt{3}$	&	$c<-2/\sqrt{3}$	&	$\va(x_1^*)>\va(x_2^*)$	&	$k<0$	&	node or focus	\\
            &		&		&	$-2/\sqrt{3}<c<0$	&	$\va(x_1^*)>\va(x_2^*)$	&	$k>0$	&	saddle	\\
            &		&		&	$c>x_1^*$	&	$\va(x_1^*)>\va(x_2^*)$	&	$k>0$	&	node or focus	\\
            &		&		&	$c<-2/\sqrt{3}$	&	$\va(x_1^*)<\va(x_2^*)$	&	$k>0$	&	saddle	\\
            &		&		&	$-2/\sqrt{3}<c<0$	&	$\va(x_1^*)<\va(x_2^*)$	&	$k<0$	&	node or focus	\\
            &		&		&	$c>x_1^*$	&	$\va(x_1^*)<\va(x_2^*)$	&	$k<0$	&	saddle	\\ 
      \end{tabular}
   \end{center}
   \caption{ Sufficient conditions on $c$ for the existence of a folded equilibrium of \eqref{eq:slowx2} with $b=0$ on  the components  $\partial A_i^*$, $i=1,\ldots,4$ of the boundary  of the attracting region of the critical manifold and classification of the folded equilibrium. 
   The regions are those of Figure~\ref{fig:VarCrit}.}
   \label{table:folded}
\end{table}

In the two examples shown in Figures~\ref{fig:MMO1} and \ref{fig:MMO2}, the trajectory goes near a folded node in the boundary of region $A_3$.
The canard persists as a high frequency oscillation of small amplitude that alternates with the large amplitude relaxation oscillation of lower frequency in a {\em mixed-mode oscillation}.
The small oscillations remain close to the synchrony plane while the large ones make alternate visits to the regions $A_2$ and $A_4$.
The small oscillations take place close to a double fold point, that we proceed to discuss.

\begin{figure}
   \parbox{0.3\linewidth}{\begin{center}
      \includegraphics[width=\linewidth]{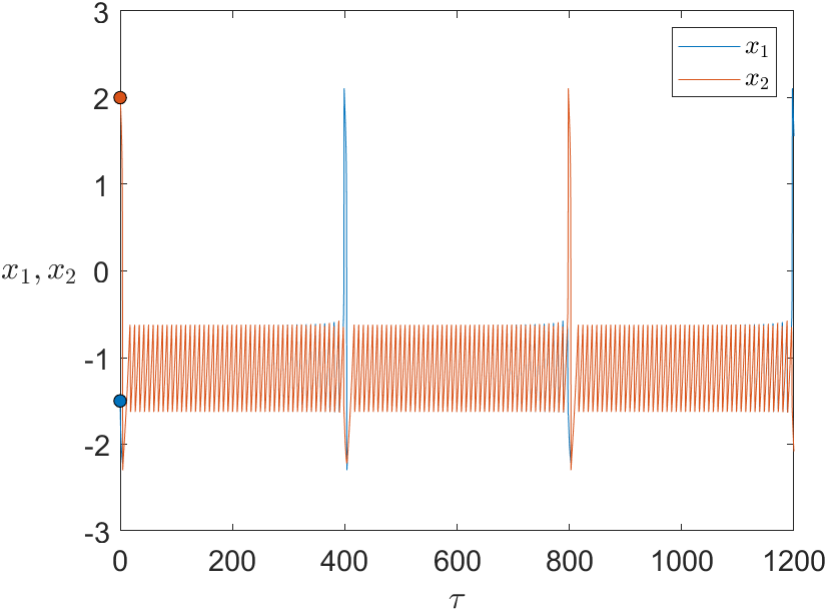}
		\\ (A)
   \end{center}}\quad
   \parbox{0.3\linewidth}{\begin{center}
		\includegraphics[width=\linewidth]{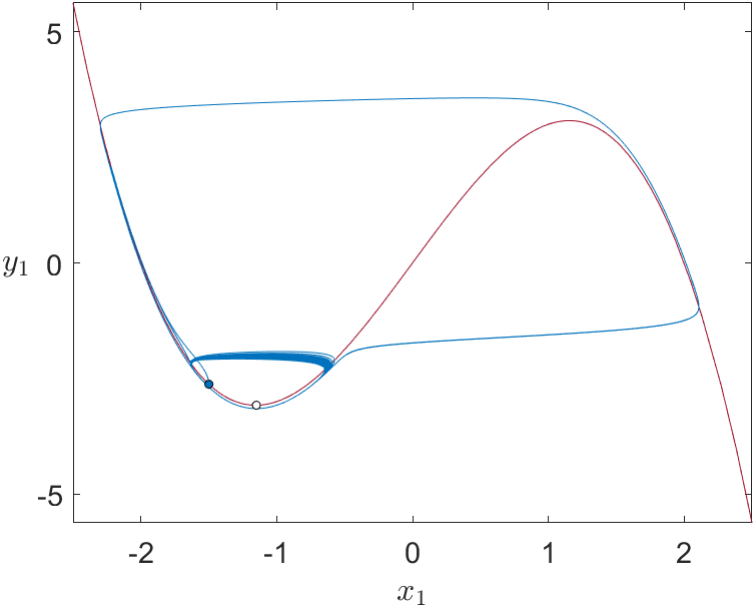}
		\\ (B)
   \end{center}}\quad
   \parbox{0.3\linewidth}{\begin{center}
      \includegraphics[width=\linewidth]{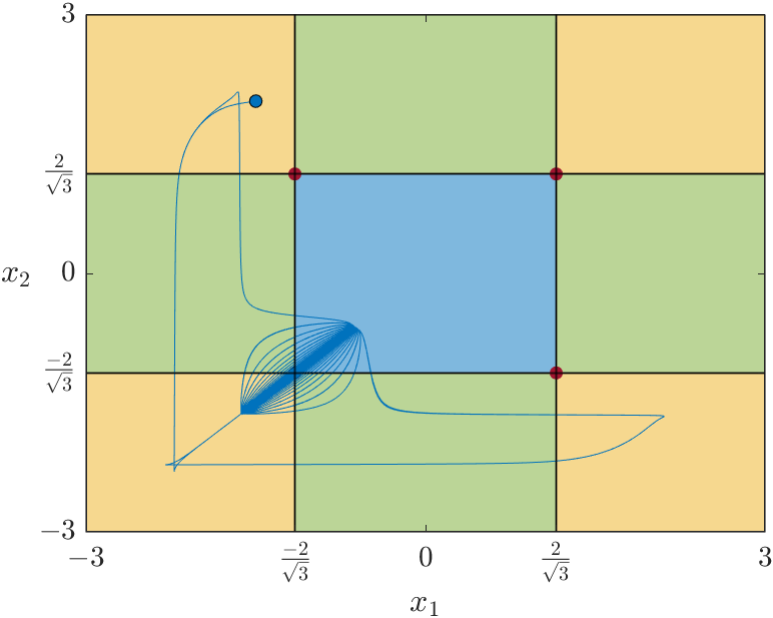}
      \\ (C)
   \end{center}}
	\\
   \parbox{0.3\linewidth}{\begin{center}
      \includegraphics[width=\linewidth]{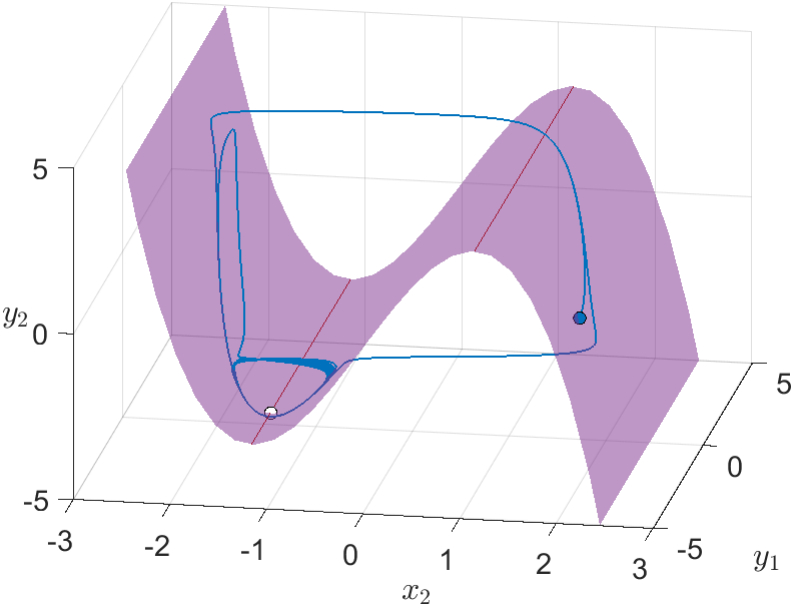}
		\\ (D)
   \end{center}}\quad
   \parbox{0.3\linewidth}{\begin{center}
		\includegraphics[width=\linewidth]{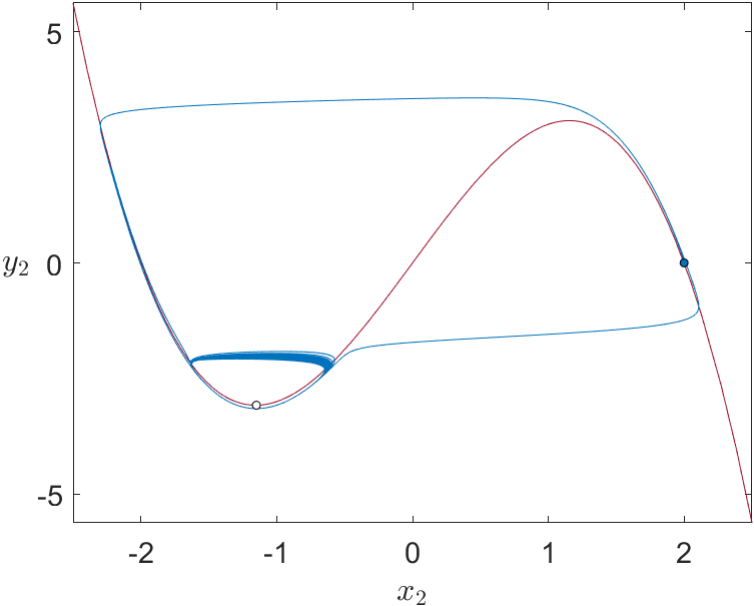}
		\\ (E)
   \end{center}}\quad
   \parbox{0.3\linewidth}{\begin{center}
		\includegraphics[width=\linewidth]{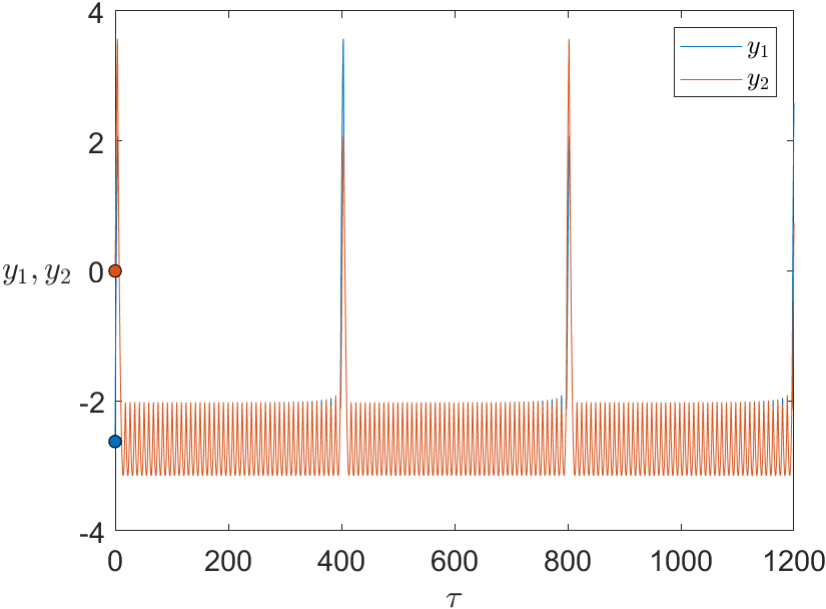}
		\\ (F)
   \end{center}}
	\caption{\small Illustration of Theorem~\ref{teo:doubleFb0}. 
            Mixed-mode oscillations arising from a canard on a solution of \eqref{eq:2FHN} near a folded node, parameters $b=0$, $c=-1.150079575$, $k=1$ and $\varepsilon=0.5$, initial condition $\left(x_{1},x_{2},y_{1},y_{2}\right)=\left(-1.5,2,\va(-1.5),\va(2)\right)$. These numerics illustrate the combination of large and small oscillations in a persistent and regular way.
            (A) - Time course for  $x_1(t)$ (blue) and $x_2(t)$ (orange). 
            (B) - Projection of the trajectory (blue) on the $\left(x_1,y_1\right)$ plane, initial condition on the blue dot,  red critical manifold $C_0$.  
            (C) - Projection of the trajectory (red) on the $\left(x_1,x_2\right)$ plane, initial condition on the blue dot.
            (D) - Projection of the trajectory (blue) on the $\left(x_2,y_1,y_2\right)$ space, initial condition on the blue dot, purple critical manifold $C_0$. 
            (E) - Projection of the trajectory (blue) on the $\left(x_2,y_2\right)$ plane, conventions as in (B), white dot folded equilibrium. 
            (F) - Time course for  $y_1(t)$ (blue) and $y_2(t)$ (orange).
         }
	\label{fig:MMO1}
\end{figure}

\begin{figure}
   \parbox{0.3\linewidth}{\begin{center}
      \includegraphics[width=\linewidth]{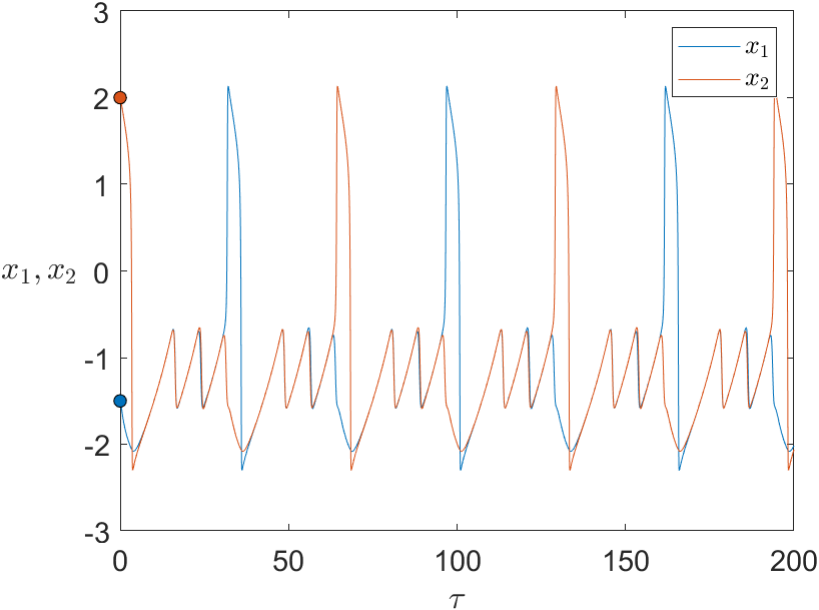}
		\\ (A)
   \end{center}}\quad
   \parbox{0.3\linewidth}{\begin{center}
		\includegraphics[width=\linewidth]{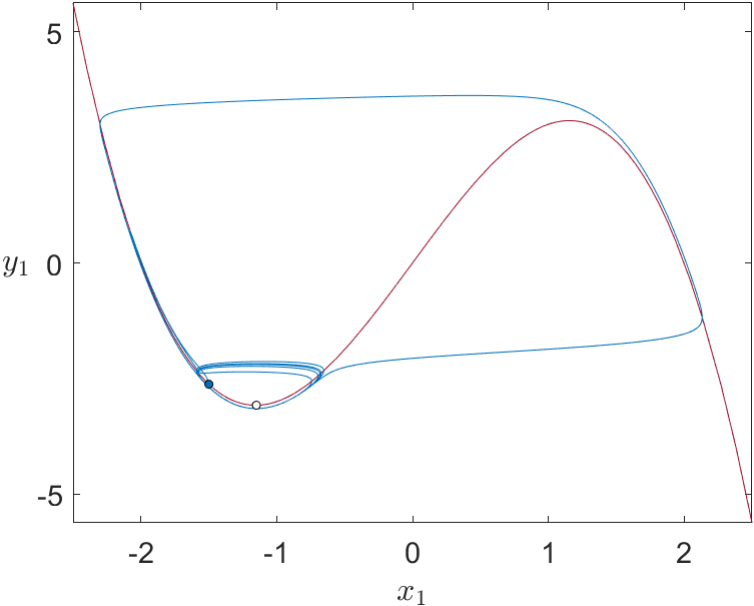}
		\\ (B)
   \end{center}}\quad
   \parbox{0.3\linewidth}{\begin{center}
      \includegraphics[width=\linewidth]{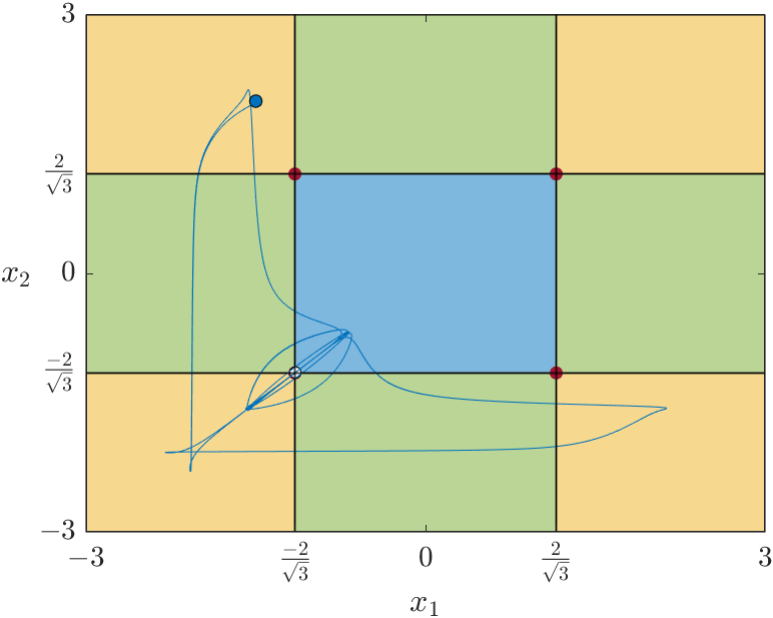}
      \\ (C)
   \end{center}}
	\\
   \parbox{0.3\linewidth}{\begin{center}
		\includegraphics[width=\linewidth]{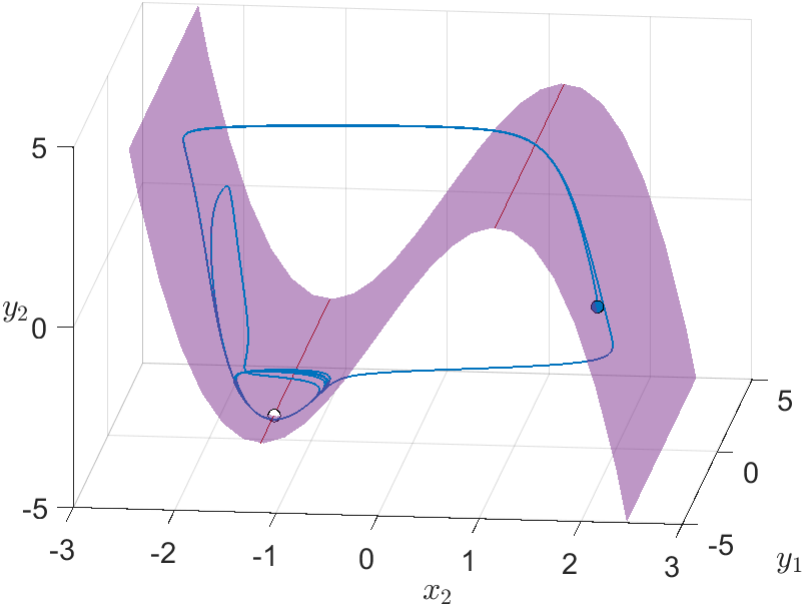}
		\\ (D)
   \end{center}}\quad
   \parbox{0.3\linewidth}{\begin{center}
		\includegraphics[width=\linewidth]{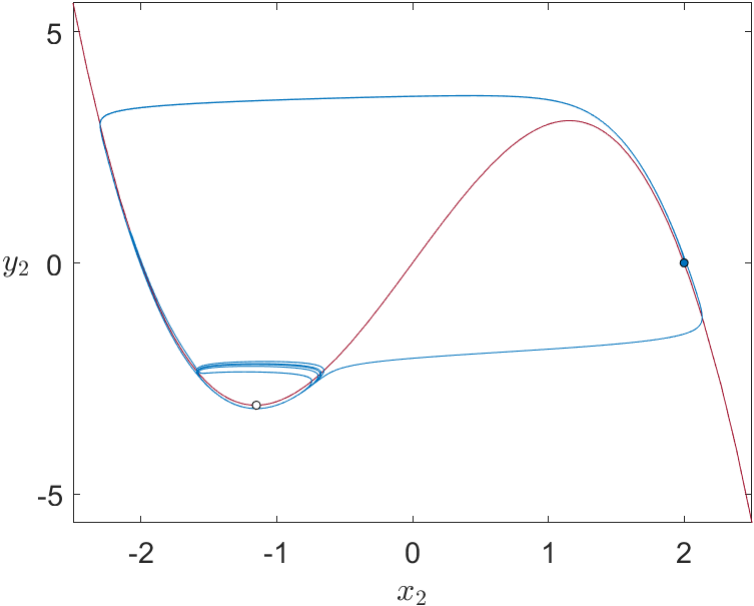}
		\\ (E)
   \end{center}}\quad
   \parbox{0.3\linewidth}{\begin{center}
		\includegraphics[width=\linewidth]{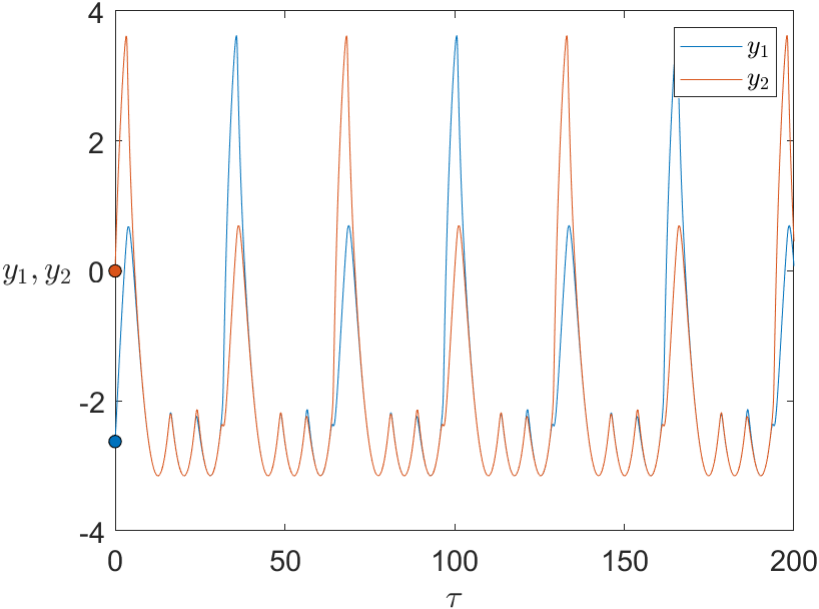}
		\\ (F)
   \end{center}}
	\caption{\small Illustration of Theorem~\ref{teo:doubleFb0}. 
            Mixed mode oscillations arising from a canard on a solution of \eqref{eq:2FHN} near a folded node, parameters $b=0$, $c=-1.1501075$, $k=0.5$ and $\varepsilon=0.5$, initial condition $\left(x_{1},x_{2},y_{1},y_{2}\right)=\left(-1.5,2,\va(-1.5),\va(2)\right)$. 
            (A) - Time course for $x_1(t)$ (blue) and $x_2(t)$ (orange). 
            (B) - Projection of the trajectory (blue) on the $\left(x_1,y_1\right)$ plane, initial condition on the blue dot, red critical manifold $C_0$.  
            (C) - Projection of the trajectory (blue) on the $\left(x_1,x_2\right)$ plane, initial condition on the blue dot.
            (D) - Projection of the trajectory (blue) on the $\left(x_2,y_1,y_2\right)$ space, initial condition on the blue dot, purple critical manifold $C_0$. 
            (E) - Projection of the trajectory (blue) on the $\left(x_2,y_2\right)$ plane, conventions as in (B). 
            (F) - Time course for $y_1(t)$ (blue) and $y_2(t)$ (orange).
         }
	\label{fig:MMO2}
\end{figure}

\subsection{Double fold points}\label{subsec:doublefolds}
\phantom{lixo}

The treatment of folded equilibria of Subsection~\ref{sec:foldedEq} above cannot be applied around the double fold points where two fold lines cross. 
This is because at these points both $\va'(x_1^*)=0$ and $\va'(x_2^*)=0$ and hence the first equation in \eqref{eq:slowx2} is not well defined, so these points have to be analysed separately.
In this section we deal with the dynamics around the double fold points $X=\left(x_1^*,x_2^*\right)=\left(x^*,\pm x^*\right)$, $x^*= 2\sigma/\sqrt{3}$, $\sigma=\pm 1$.

We start with the equations \eqref{eq:slowx1} and introduce a time rescaling of $\tau=t/\va'(x_1^*)\va'(x_2^*)$, that is singular at $x_i=x_i^*$, $i=1,2$.
Then the equations \eqref{eq:slowx1} transform into
\begin{equation}\label{eq:slowx3}
   \left\{\begin{array}{lclcl}
   x_1'&=&\va'(x_2)\left(
   x_1-(b+k)\va(x_1)+k\va(x_2)-c\right)&=&F_1(x_1,x_2)\\
   x_2'&=&\va'(x_1)\left(x_2-(b+k)\va(x_2)+k\va(x_1)-c\right)&=&F_2(x_1,x_2) \ .
   \end{array}\right.
\end{equation}
Note that $F_2(x_1,x_2)=F_1(x_2,x_1)$. 
In what follows, let $F(x_1,x_2)=\left(F_1(x_1,x_2),F_2(x_1,x_2)\right)$, be defined in $(x_1, x_2)\in \RR^2\backslash\{(x_1, x_2): \varphi'(x_1)\varphi'(x_2)=0\}$.

Recall that we are using  the notation $A_i^*$ for the sets  
$$
A_i^*=\left\{(x_1,x_2) \in\RR^2:\ \left(x_1,x_2,\va(x_1),\va(x_2)\right)\in A_i\right\}, \quad i=1,\ldots,4,
$$  
we make the analogous convention for $S_i^*$, $i=1,\ldots,4$ and $R^*$ (see Figure~\ref{fig:VarCrit}).
For $(x_1,x_2)\in A_i^*$ we have $\va'(x_j)<0$, $j=1,2$ hence the time rescaling preserves time orientation inside the $A_i^*$. 
Since $\va'(x)>0$ for $-2/\sqrt{3}<x<2/\sqrt{3}$ then the time rescaling also preserves time orientation in $R^*$ and reverses it in the $S_i^*$, $i=1,\ldots,4$.
The main result of this section is the following:

\begin{theorem}\label{teo:doubleFold}
   Let $X=\left(x^*,\pm x^*\right)$, $x^*= 2\sigma/\sqrt{3}$, $\sigma=\pm 1$ and let $\phi(X)=\left(\va(x^*),\pm\va(x^*)\right)$.
   It is possible to have a canard for \eqref{eq:2FHN} near the double fold point $\left(X,\phi(X)\right)$ under the following conditions on the parameters of \eqref{eq:2FHN}:
   \begin{enumerate}
      \renewcommand{\theenumi}{\roman{enumi}}
      \renewcommand{\labelenumi}{({\theenumi})}
      \item\label{item:canardA13} for $X=\left(x^*,x^*\right)$ (in either $\partial A_1^*$ or $\partial A_3^*$) if $\sigma c<2/\sqrt{3}-b\va(2/\sqrt{3})$;
      \item\label{item:canardA24} for $X=\left(x^*,-x^*\right)$ (in either $\partial A_2^*$ or $\partial A_4^*$) if $b+2k<3/8$ and $|c|< 2/\sqrt{3}-(b+2k)\va(2/\sqrt{3})$.
      \setcounter{lixo}{\value{enumi}}
   \end{enumerate}
   There are no canards in a neighbourhood of the double fold point $\left(X,\phi(X)\right)$ 
   under the following conditions on the parameters of \eqref{eq:2FHN}:
   \begin{enumerate}
      \renewcommand{\theenumi}{\roman{enumi}}
      \renewcommand{\labelenumi}{({\theenumi})}
      \setcounter{enumi}{\value{lixo}}
      \item\label{item:noCanardA13} for $X=\left(x^*,x^*\right)$ (in either $\partial A_1^*$ or $\partial A_3^*$) if $\sigma c>2/\sqrt{3}-b\va(2/\sqrt{3})$;
      \item\label{item:noCanardA24} for $X=\left(x^*,-x^*\right)$ (in either $\partial A_2^*$ or $\partial A_4^*$) if $b+2k>3/8$ and $|c|<(b+2k)\va(2/\sqrt{3})-2/\sqrt{3}$.
   \end{enumerate}
\end{theorem}

The parameters for the mixed-mode oscillations in Figures~\ref{fig:MMO1} and \ref{fig:MMO2} satisfy condition \eqref{item:canardA13}.
Note that condition \eqref{item:canardA24} only holds if $2/\sqrt{3}-(b+2k)\va(2/\sqrt{3})>0$ and that \eqref{item:noCanardA24} needs that $2/\sqrt{3}-(b+2k)\va(2/\sqrt{3})<0$.

A canard exists if there is a trajectory of \eqref{eq:2FHN} starting on the attracting part of the slow manifold $C_\varepsilon$ that crosses the fold line into the part of $C_\varepsilon$ that is not attracting.
The idea of the proof of Theorem~\ref{teo:doubleFold} is to show that under any of the conditions \eqref{item:canardA13} and \eqref{item:canardA24} the point $X$ is a saddle for $F$ and its stable manifold intersects the corresponding $A_i^*$, creating the possibility of canards near $X$, as in \cite[Lemma 3.1]{SW01}. 
For conditions \eqref{item:noCanardA13} and \eqref{item:noCanardA24} the idea is to show that although the point $X$ is also a saddle, its stable manifold does not intersect the corresponding $A_i^*$, while its unstable manifold does. Hence canards are not possible around that point.

\begin{proof}
   Let $X=\left(x^*,\pm x^*\right)$, $x^*= 2\sigma/\sqrt{3}$, $\sigma=\pm 1$.
   Since $\va'(x^*)=0$, then $\dfrac{\partial F_1}{\partial x_1}(X)=\dfrac{\partial F_2}{\partial x_2}(X)=0$.
   Therefore, $\tr DF(X)=0$ and hence, unless $\det DF(X)=0$,  the point $X$ is either a saddle or a centre for the linearisation of the rescaled equations \eqref{eq:slowx3}.
   We treat separately the two cases $X=\left(x^*, x^*\right)$ and $X=\left(x^*,- x^*\right)$, $x^*= 2\sigma/\sqrt{3}$, $\sigma=\pm 1$.

   \subsection*{First case}
   First, let $X=\left(x^*,x^*\right)$, $x^*= 2\sigma/\sqrt{3}$, $\sigma=\pm 1$ be the double folded equilibrium of \eqref{eq:slowx3}, lying in either $\partial A_1^*$ or $\partial A_3^*$.
   The derivative $DF(X)$ is given by
   \begin{equation}\label{eq:DFX}
      DF(X)=\begin{pmatrix}
      0
      &\va''(x^*)(x^*-b\va(x^*)-c)\\
      \va''(x^*)(x^*-b\va(x^*)-c)
      &0
      \end{pmatrix} .
   \end{equation}
   The matrix $DF(X)$ is symmetric and hence its eigenvalues are real.
   Also
   $$
   \det DF(X)=-\left(\va''(x^*)\right)^2(x^*-b\va(x^*)-c)^2\le 0
   $$ 
   and hence, unless $c=x^*-b\va(x^*)$, the point $X$ is a saddle.
   The eigenvalues of the matrix $DF(X)$ are $\lambda_1= \va''(x^*)(x^*-b\va(x^*)-c)$ and $\lambda_2=-\lambda_1$ with eigenspaces $$V_1=\left\{\left(s,s\right)\ s\in\RR\right\} \quad \text{and} \quad V_2=\left\{\left(s, -s\right)\ s\in\RR\right\}$$, respectively.
   Thus locally one of the trajectories of \eqref{eq:slowx3} that is tangent to $V_1$ lies in $ A_1^*$ or $ A_3^*$, according to the case in question and all the trajectories of \eqref{eq:slowx3} that are tangent to $V_2$ lie outside $ A_1^*$ and $ A_3^*$.

   Therefore, if $\lambda_1<0$ the stable manifold of $X$ intersects the corresponding $ A_j^*$, $j=1,3$ and canards are possible.
   Since for $x^*=2\sigma/\sqrt{3}$, $\sigma=\pm 1$ the sign of $\va''(x^*)$ is that of $-\sigma$, then $\lambda_1<0$ if and only if $\sigma c<2/\sqrt{3}- b\va(2/\sqrt{3})$.
   In this case then the trajectory tangent to $V_1$ that lies in $A_i^*$, $i=1,3$ goes into $X$, so a canard is possible and assertion \eqref{item:canardA13} follows.
   
   On the other hand, if $\lambda_1>0$, then the trajectory tangent to $V_1$ that lies in $A_i^*$, $i=1,3$ goes to the interior of $A_i^*$ in positive time.
   Moreover, at all points near $X$ near the boundary $\partial A_i^*$ the vector field associated to \eqref{eq:slowx3} points into the interior of $A_i^*$.
   Since in both $A_1^*$ and $A_3^*$ the time rescaling preserves time orientation, there cannot be a canard proving assertion \eqref{item:noCanardA13}.

   \subsection*{Second case}
   Now we address the case when the double folded equilibrium of \eqref{eq:slowx3} is $X=(x^*,-x^*)$, $x^*= 2\sigma/\sqrt{3}$, $\sigma=\pm 1$, lying in either $\partial A_2^*$ or $\partial A_4^*$.
   The matrix of the derivative $DF(X)$ at $X=(x^*,-x^*)$ is not symmetric.
   It is given by
   \begin{equation}\label{eq:DFxtil}
      DF(X)=\begin{pmatrix}
      0&
      -\va''(x^*)(x^*-(b+2k)\va(x^*)-c)
      \\
      -\va''(x^*)(x^*-(b+2k)\va(x^*)+c)&0
      \end{pmatrix} 
   \end{equation}
   hence $\tr DF(X)=0$, and $\det DF(X)=-{\va''(x^*)}^2\left[ (x^*-(b+2k)\va(x^*))^2-c^2\right]$.
   If $\det DF(X)>0$, then $DF(X)$ has a pair of purely imaginary eigenvalues.
   The point $X$ is a centre for the linearisation of the desingularised equations \eqref{eq:slowx3}.
   This happens when $|c|>\left|(b+2k)\va(x^*)-x^*\right|$, we do not include this situation in our analysis.

   If $\det DF(X)<0$, then the point $X$ is a saddle, since the eigenvalues of $DF(X)$ are $\lambda_\pm=\pm\sqrt{-\det DF(X)}$.
   This happens if and only if 
   \begin{equation}\label{eq:condSaddle}
      |c|<\left| x^*-(b+2k)\va(x^*)\right|=\left|  2/\sqrt{3}-(b+2k)\va( 2/\sqrt{3})\right|
   \end{equation}
   and in this case the two non-zero entries in $DF(X)$ have the same sign.

   Using the same arguments of the case $X=(x^*,x^*)$, it is possible to have canards if one branch of the stable manifold of $X$ lies in $A_i^*$, $i=2,4$.
   \emph{Canards} are not possible if the condition fails.
   This is determined by the directions of the eigenspace associated to the negative eigenvalue $\lambda_-$, which is the set $V_-=\{(s(\alpha,\lambda_-): s\in\RR\}$ where $$\alpha=-\va''(x^*)(x^*-(b+2k)\va(x^*)-c)=\dfrac{\partial F_1}{\partial x_2}(X).$$
   There are no canards if $\alpha<0$, they are possible if $\alpha>0$.

   It remains to show that the conditions $\alpha>0$ and \eqref{eq:condSaddle} are equivalent to assertion \eqref{item:canardA24} in the statement and that $\alpha<0$ and \eqref{eq:condSaddle} are equivalent to the assertion \eqref{item:noCanardA24}.

   To do this, let $\beta=2/\sqrt{3}-(b+2k)\va( 2/\sqrt{3})$, and note that $\alpha=-\va''(2/\sqrt{3})\left(\beta-\sigma c\right)$. 
   Since $\va''(2/\sqrt{3})=-12/\sqrt{3}<0$ then $\alpha$ has the same sign as $\beta-\sigma c$.
   Also note that because $\sigma =\pm 1$ then condition \eqref{eq:condSaddle} is equivalent to $|\sigma c|<|\beta|$.

   If $\beta>0$, then condition \eqref{eq:condSaddle} is equivalent to $-\beta<\sigma c<\beta$ and this implies that $\alpha>0$.  
   Since $\va(2/\sqrt{3})=16/3\sqrt{3}$ then $\beta>0$ if and only if $b+2k<3/8$.

   If $\beta<0$, then condition \eqref{eq:condSaddle} is equivalent to $\beta<\sigma c<-\beta$ and this implies that $\alpha<0$. 
   Also $\beta<0$ if and only if $b+2k>3/8$.

   We have established that $\alpha>0$ (implying canards are possible) if and only if \eqref{eq:condSaddle} holds and $\beta>0$. 
   The necessary conditions are that $b+2k<3/8$ and $|c|<|\beta|=\beta$, as in assertion \eqref{item:canardA24}.
   We have also established that $\alpha<0$ (implying canards are not possible) if and only if \eqref{eq:condSaddle} holds and $\beta<0$, and the necessary conditions are that $b+2k>3/8$ and $|c|<|\beta|=-\beta$, as in assertion \eqref{item:noCanardA24}.
\end{proof}

In the first part of Theorem~\ref{teo:doubleFold} we only claim that \emph{canards} are possible because it is not clear that trajectories of \eqref{eq:slowx3} that tend to the boundary of the $A_i^*$ correspond to trajectories of \eqref{eq:2FHN} that continue into $C_0\backslash A$.
This is because the time rescaling reverts time orientation in the $S_i$, the components of $C_0$ where the equilibria of the fast equation are saddles, as shown in Figure~\ref{fig:DFA1-}.
Deciding if canards exist in each case requires a detailed analysis that is beyond the scope of this article, as mentioned in Section \ref{sec:discussion}.
In the special case $b=0$ the results of Subsection~\ref{subsec:b0} may be used to improve the result to cover the cases shown in Figures~\ref{fig:MMO1} and \ref{fig:MMO2}.

\begin{figure}
   \parbox{.35\linewidth}{
   \begin{center}
      \includegraphics[width=\linewidth]{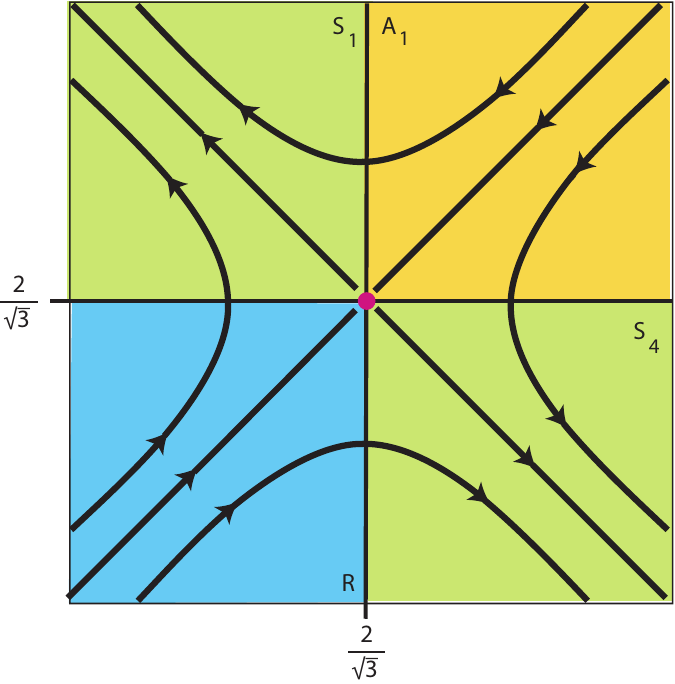}\\
      time $\tau$
   \end{center}
   }
   \qquad
   \parbox{.35\linewidth}{
   \begin{center}
      \includegraphics[width=\linewidth]{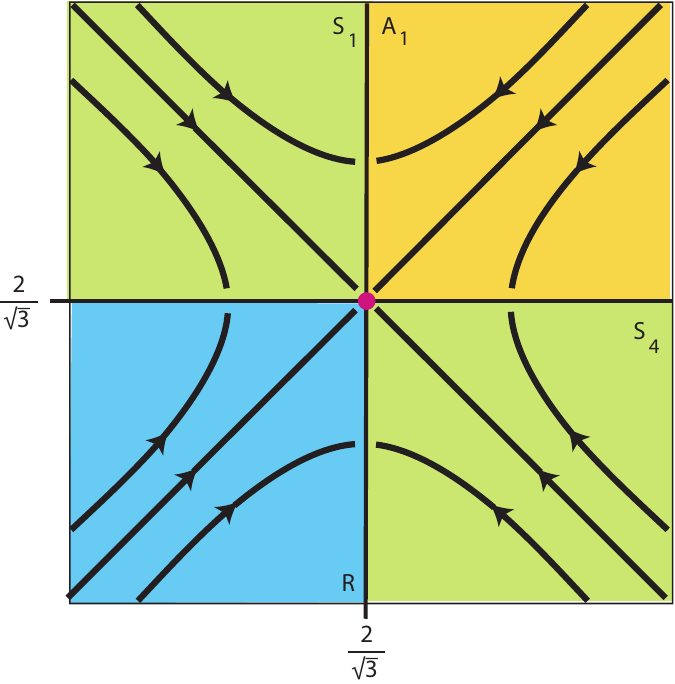}\\
      time $t$
   \end{center}
   }
   \caption{\small On the left, trajectories of \eqref{eq:slowx3} on the $\left(x_1,x_2\right)$ plane, near the double fold point $X=\left(2/\sqrt{3},2/\sqrt{3}\right)\in\partial A_1^*$ for $\lambda_1<0$, in the rescaled time $\tau$.
            On the right, trajectories of \eqref{eq:slowx1} near the same point, in the original time $t$. 
            The rescaling inverts the time orientation in regions $S_j$ (green), $j=1,\ldots,4$ and preserves it in regions $R$ (blue) and $A_j$ (yellow), $j=1,\ldots,4$. 
            The dynamics of \eqref{eq:slowx1} is not well defined when $x_i=2/\sqrt{3}$, $i=1,2$.
         }
   \label{fig:DFA1-}
\end{figure}

\begin{theorem}\label{teo:doubleFb0}
   Let $X=\left(x^*,x^*\right)$, $x^*= 2\sigma/\sqrt{3}$, $\sigma=\pm 1$ be a double fold point.
   If $b=0$, $k>0$ and $\sigma c<2/\sqrt{3}$, then for $\sigma c$ close to $2/\sqrt{3}$ and small $\varepsilon>0$ there is a canard for \eqref{eq:2FHN} close to $X$.
\end{theorem}

\begin{proof}
   In Table~\ref{table:folded} the conditions of the statement correspond to the first rows in the regions $\partial A_1^*$ and $\partial A_3^*$, where the point $\left(x_1^*,x_2^*\right)$, with $x_2=x^*$ and $\sigma x_1^*>2/\sqrt{3}$ is either a node or a focus.
   The idea of the proof is to show  that:
   \begin{enumerate}
      \renewcommand{\theenumi}{(\alph{enumi})}
      \renewcommand{\labelenumi}{{\theenumi}}
      \item\label{item:a} as $c\to 2\sigma/\sqrt{3}$ these points accumulate on $X$, and 
      \item\label{item:b} for $\sigma c$ close to $2/\sqrt{3}$  the point is a node.
   \end{enumerate}
   Then it will follow, either by Lemma 2.3 of \cite{SW01}, or by the results of \cite{GuckenheimerHaduc} and \cite{Guckenheimer}, that there are solutions of \eqref{eq:2FHN} starting close to $X$ in the attracting part of $C_\varepsilon$ that follow the repelling part of $C_\varepsilon$ for some time.

   For \ref{item:a} we use the fact that the restriction of the function $\va(x)$ to $\sigma x>2/\sqrt{3}$ is a diffeomorphism.
   Therefore, for any $\delta_1>0$ there exists a $\delta_2>0$ such that if $\sigma x_2>2/\sqrt{3}$ and $\left| \va(x_1^*)-\va(x_2^*) \right|=\delta_2$, then $\left|x_1^*-x_2^*\right|<\delta_1$.
   Using \eqref{eq:dva}, if $|c-x_2^*|<k\delta_2$  then $\left| \va(x_1^*)-\va(x_2^*) \right|<\delta_2$ and the result follows.

   For \ref{item:b} we have already established in Subsection~\ref{subsec:b0} that in this case $\det DH\left(x_1^*,x_2^*\right)>0$.
   From the expression \eqref{eq:detDHb0} and assertion (a) it follows that when $c$ tends to $2\sigma/\sqrt{3}$ then $\det DH\left(x_1^*,x_2^*\right)$ tends to 0. 
   Therefore, for $c$ close to $2\sigma/\sqrt{3}$ we have $0<\det DH\left(x_1^*,x_2^*\right)<1/4$ and since $\tr DH\left(x_1^*,x_2^*\right)=1$, the point $\left(x_1^*,x_2^*\right)$ is a node.
\end{proof}

\section{Discussion}\label{sec:discussion}
This article is an analysis of the dynamics of two identical FitzHugh--Nagumo equations symmetrically coupled through the slow equations. 
In this section we compare our results to findings by other authors using different types of coupling.
Most authors only deal with the fast-to-fast coupling, probably motivated by its neurophysiological interpretation as an electrical synapse.
The coupling described in the present article does not have a natural interpretation in neurophysiology but can be realised experimentally in electrical circuits, as is  the case of the simultaneous slow-to-slow and fast-to-fast connections in Kawato {\sl et al.} \cite{KawatoEtAl1979}.
We have focused on {\em synchrony}, {\em bistability} and on {\em mixed-mode oscillations} generated by {\em canards}.

\subsection{Synchrony and bistability}
\phantom{lixo}

When two identical equations are coupled symmetrically there are always synchronous solutions, as is immediate from the symmetry. 
The interesting questions are whether the synchrony space contains non trivial dynamics, whether the  synchronous solutions are asymptotically stable and hence visible, and whether approximate synchrony persists when the symmetry is broken.
Similarly, coupling equations that have some additional symmetry always yields antisynchronous solutions to which analogous comments are applicable.
Section~\ref{sec:synchro} of this article is concerned with providing an answer to these questions in the present context.
\bigbreak

We use the symmetry of the coupled equations to obtain in Theorem~\ref{teo:synchro} attracting persistent synchronous solutions for all values of the parameters $b$ and $c$ in FHN and for sufficiently large coupling strengh $k$.
In particular we show in Corollary~\ref{cor:Relax} that if $b>1/4$ there are attracting synchronous periodic solutions even with negative coupling strength.
We also show in Corollary \ref{cor:synchroPersists} that approximate synchrony will persist if the symmetry is broken either by taking slightly different coupling constants or by a small change in the equations governing one of the cells.

Synchronised periodic solutions were also observed analytically by Campbell and Waite \cite{CampbellWaite} and numerically in passing by Hoff {\sl et al.} \cite{HoffEtAl2014} when two FHN were coupled symmetrically and bidirectionally through the fast equations, but the latter do not report on persistence under symmetry breaking, their focus being on chaotic solutions.
Kawato  {\sl et al.} \cite{KawatoEtAl1979} obtain both the synchrony and its persistence under symmetry breaking perturbations for two FHN with simultaneous slow-to-slow and fast-to-fast connection.
Pedersen {\sl et al.} \cite{PedersenEtal2022} provide bifurcation diagrams in the synchrony plane for two FHN  coupled symmetrically and bidirectionally through the fast equations, showing steady-state and Hopf bifurcations that create both stable and unstable limit cycles.
Experimental results on synchrony of a radio engineering circuit that simulates two FHN with different values of the parameter $c$ coupled through the fast equations with time delay are presented in both Ponomarenko {\sl et al.} \cite{Ponomarenko} and Kulminskiy  {\sl et al.} \cite{Kulminskiy}, with the latter also containing numerical simulations. 
Both papers observe synchrony both in the symmetric case and in a large region in parameter space where the symmetry is broken. 
However, since their connection has a time delay, most of their findings are more similar to the antisynchrony situation discussed below.

Perfect and approximate synchrony are obtained analytically by Plotnikov and Fradkov \cite{Plotnikov} for two FHN differing on the parameter $c$ and coupled through the fast equations.
They use a Lyapunov function to estimate the precision of synchronisation (cf. Definition~\ref{def:appsynchro}) for large differences in the two values of the parameter $c$.
For two identical FHN they also obtain perfect synchrony.

It is clear from all the results discussed above and ours that asymptotically stable synchrony arises from all types of linear coupling, even very far from the symmetric situation.
An interesting future task is to extend the analysis done by Plotnikov and Fradkov \cite{Plotnikov} to a complete bifurcation study of all the types of coupling and of symmetry breaking.
\bigbreak

When the parameter $c$ is zero the individual FitzHugh--Nagumo equations becomes symmetrical under change of sign of both variables.
This provides an additional symmetry to the coupled equations, with a flow-invariant fixed point subspace. 
From this we obtain Theorem~\ref{teo:antisynchro} that implies the coexistence of different stable solutions.
We show that the two sets of stable solutions persist when the additional symmetry is broken for small $c\ne 0$.
The additional symmetry does not hold when the equations are coupled through the fast equations, so the bistability stated in Corollary \ref{cor:bistabilityPersists} is a characteristic feature of the type of coupling analysed here.

Antisynchrony, in the sense of oscillations with a half period difference, was observed by Ponomarenko {\sl et al.} \cite{Ponomarenko} and by Kulminskiy {\sl et al.} \cite{Kulminskiy} in experiments with two FHN electrical circuits coupled through the fast equations with a time delay.
Although their system has $c\ne 0$ and $b=0$, apparently the time delay induces a similar effect to the symmetry. 
They use a delay that seems to be close to half the period of the relaxation oscillations and the effect is observed both in two identical FHN and in the case of systems with different values of the parameter $c$ persisting for a large difference in parameter values.

Bistability has been found in the work by Kawato {\sl et al.} \cite{KawatoEtAl1979} coupling together the slow variables and the fast variables to each other in two FHN systems. 
When only the slow variables are coupled, they have obtained the coexistence of synchrony and antisynchrony through Hopf bifurcation in a parameter $I$ added to the fast equations.
They also indicate that the bistability persists when the fast variables are also coupled with a small coupling constant.

Multistability has been found not associated to antisynchrony by Campbell and Waite \cite{CampbellWaite} who considered electrical coupling of two FHN models through the fast equation. 
They show that when the magnitude of the coupling is strengthened, periodic orbits undergo several bifurcations (namely resonant Hopf-Hopf interactions) leading to the coexistence of a chaotic attractor with an attracting periodic orbit.
This indicates that there might be other interesting dynamical structures that do not arise from the symmetry, a direction for further research.
Note that neither \cite{CampbellWaite} nor \cite{KawatoEtAl1979} used the fast-slow structure.

\subsection{Mixed-mode oscillations}
\phantom{lixo}

The symmetry in the model we have analysed limits the possible dynamic outcomes, and yet it still allows a number of interesting features.
The presence of a canard induces small amplitude symmetry breaking oscillations before a solution converges to a synchronised periodic solution of large amplitude, shown in Figure~\ref{fig:canard}. 
Oscillations that alternate between different amplitudes, called {\em mixed-mode}, have been associated to {\em canards} ever since they were first described in the van der Pol equations in Beno\^it {\sl et al.} \cite{Canard}.
A detailed description of the way they appear in fast-slow systems can be found in Desroches {\sl et al.} \cite{Desroches2012} and for the forced van der Pol equations in a recent article by J.~Guckenheimer \cite{Guckenheimer2025}.

Santana {\sl et al.} \cite{SantanaEtal2021} report ``transient chaos'', a complicated and long transient, in their numerical description of two FHN with different parameter values coupled through the fast equations.
Canard-induced small amplitude transients were also found by Krisitiansen and Pedersen \cite{KristiansenPedersen} in two identical FHN coupled through the fast equations.
In their case the attracting periodic solution does not lie in the synchrony plane.
The small amplitude transients arise from a canard at a cusp point in the critical manifold.
In our case this manifold does not have cusps, so the origin of the transient oscillations is not the same, here they arise at a folded node -- Proposition~\ref{cor:opensetcanard}.

Sustained mixed-mode oscillations arising from canards, like those in Figures~\ref{fig:MMO1} and \ref{fig:MMO2}, are ubiquitous in coupled FHN.
They appear in asymmetric coupling of the fast equation to the slow one in Doss-Bachelet {\sl et al.} \cite{DFP2003} and in Krupa {\sl et al.} \cite{Krupa2012} where three different time scales are considered.
Desroches {\sl et al.} \cite{DKO2008} find them in self-coupled FHN analogous to unidirectional coupling through the fast equation.
Krupa {\sl et al.} \cite{Krupa2014} find them with both the slow equations and the fast equations coupled together, when the two FHN have different parameter values. 
It is worth to read Yu \emph{et al.} \cite[pp. 364]{yu2023canard}, where the authors investigate a Hopf bifurcation and the emergence of two types of mixed mode oscillations in a periodically-perturbed Bonhoeffer-van der Pol circuit.  
We also refer the reader to \cite{chen2022parametrically} where the fast-slow dynamics of a parametrically driven oscillator with triple-well potential has been studied. 
For some parameter values, the perturbed model combined with multiple frequency components may produce fast large oscillations and slow small oscillations clustered into the singular periodic attractor. 
Authors have shown that the critical manifold is folded with attracting and repelling branches separated by a fold.
Such folded singularities can lead to oscillation behaviours with separated amplitudes -- pairs of small-amplitude oscillations about stable manifold alternate with pairs of large-amplitude oscillations that occur due to jumps from one stable branch to the other.

By combining the dynamics within the synchrony and antisynchrony subspaces, we prove bistability in \eqref{eq:2FHN}, where two types of solution coexist as hyperbolic attractors. 
They persist under small perturbation of the parameters. 
The symmetric coupling allows a manageable analysis of the system. 
Symmetry also plays the role of an ``organizing center'' -- some hyperbolic structures in the full equivariant case persist in the asymmetric flow.
 
To the best of our knowledge this is the first time sustained mixed-mode oscillations appear in symmetric coupling of identical equations.
It is possible that this is a consequence of the bistability arising from the additional symmetry of the equations: the large relaxation oscillations make excursions close to the antisynchrony plane, while the small amplitude canard oscillations remain in a neighbourhood of the synchrony plane.
If this is the case, then it will not be found when two identical FHN are coupled symmetrically through the fast equations.
The sustained mixed-mode oscillations arise close to the point where the two fold lines cross transversely as we established in Theorem~\ref{teo:doubleFb0}. 
The general analysis of the dynamics in the neighbourhood of such a point is one of the tasks we intend to pursue in the near future.

\subsection{One directional coupling}
\phantom{lixo}

Because our results are based on normal hyperbolicity, they will persist under small symmetry breaking perturbations, as pointed out in Corollaries~\ref{cor:bistabilityPersists} and \ref{cor:synchroPersists}.
However, asymmetric coupling, where one FHN is forcing the other, as in
\begin{equation}\label{eq:asym}
   \left\{
   \begin{array}{rcl}
               \varepsilon\dot{x}_1 &=& -y_1+\va(x_1)\\
               \varepsilon\dot{x}_2 &=&-y_2+\va(x_2)\\
               \dot{y}_1 &=&x_1-by_1-c +ky_2\\
               \dot{y}_2 &=& x_2-by_2-c	
   \end{array}
   \right.
   \qquad
   \va(x)=4x-x^3
   \qquad b,c,k\in \RR\quad k\ne 0 \ 
\end{equation}
is not part of this scenario.
Numerical evaluation of the Lyapunov spectrum of \eqref{eq:asym}, presented in \cite{HoffEtAl2014} for unidirectional coupling of two FHN through the fast equation, finds regions in the parameter space with frequency-locked solutions and other regions with chaotic behaviour.
In Figure~\ref{fig:chaoticMMO} we show a numerical solution of \eqref{eq:asym} that looks like a chaotic mixed-mode oscillation.
This is another research direction we intend to pursue, but it is beyond the scope of the present article.
\bigbreak

After the present study, we hope this work improves the research about theoretical and experimental applications of coupled FHN models. 
Possible directions should consider the importance of transient dynamics and the possibility of finding multiple attractors coexisting for the same parameter combinations.  
Additionally, one potential natural extension of this work is the investigation of the persistence of transient chaotic dynamics and multistability when several identical FHN models are coupled through different schemes. 

\begin{figure}
   \parbox{0.3\linewidth}{\begin{center}
      \includegraphics[width=\linewidth]{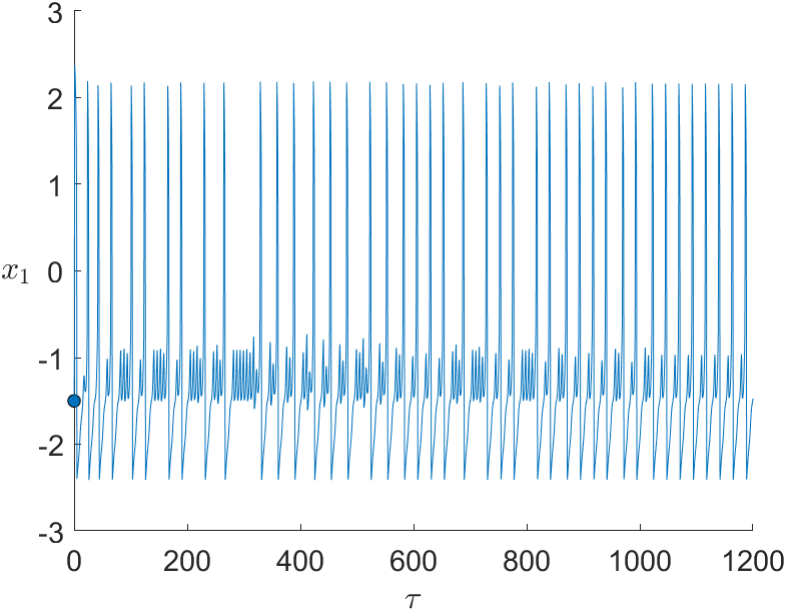}
      \\ (A)
   \end{center}}\quad
   \parbox{0.3\linewidth}{\begin{center}
      \includegraphics[width=\linewidth]{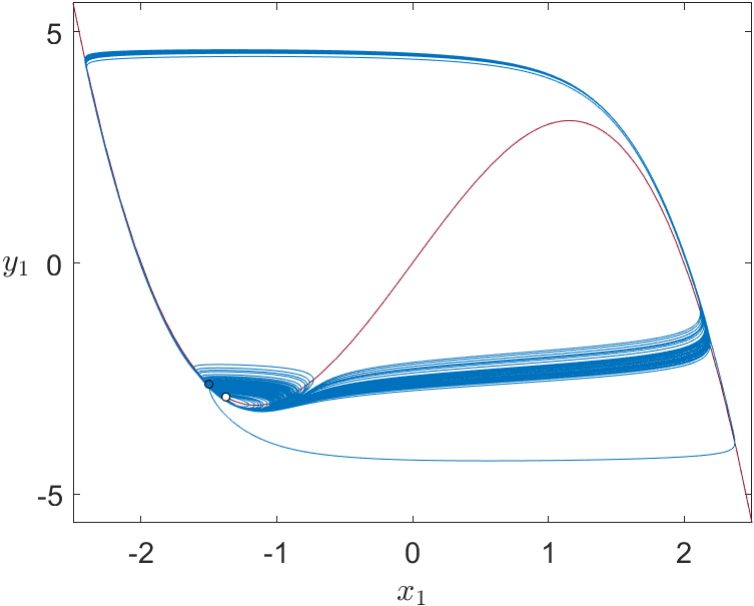}
      \\ (B)
   \end{center}}\quad
   \parbox{0.3\linewidth}{\begin{center}
      \includegraphics[width=\linewidth]{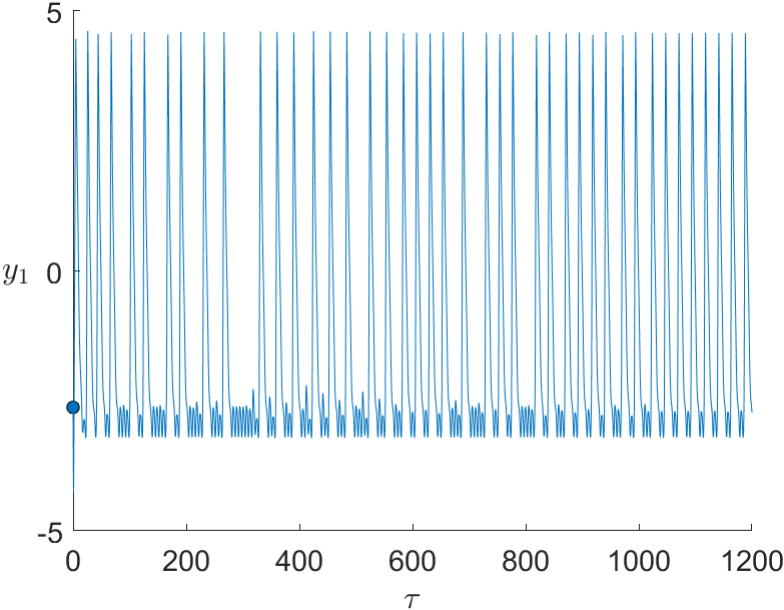}
      \\ (C)
   \end{center}}
   \\
   \parbox{0.3\linewidth}{\begin{center}
      \includegraphics[width=\linewidth]{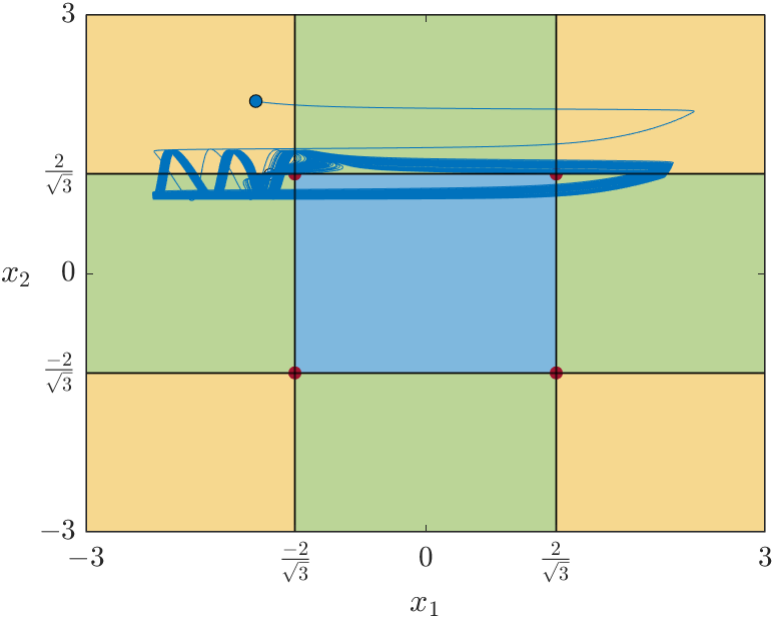}
      \\ (D)
   \end{center}}\quad
   \parbox{0.3\linewidth}{\begin{center}
      \includegraphics[width=\linewidth]{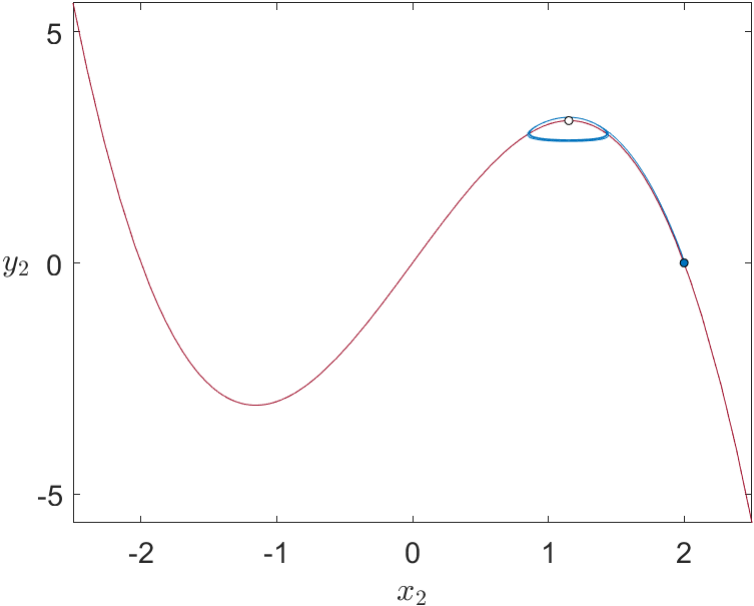}
      \\ (E)
   \end{center}}\quad
   \parbox{0.3\linewidth}{\begin{center}
      \includegraphics[width=\linewidth]{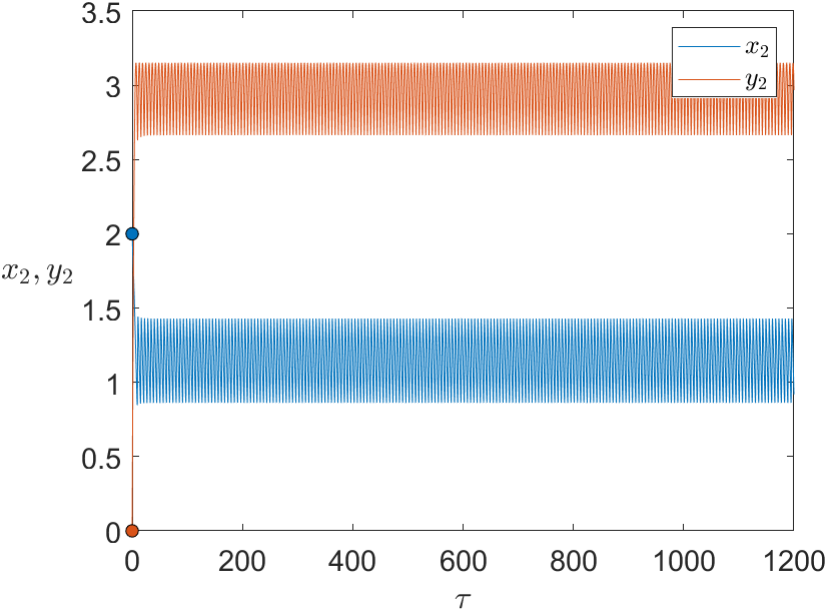}
      \\ (F)
   \end{center}}
   \caption{\small Chaotic mixed-mode oscillations in \eqref{eq:asym}, parameters $b=0$, $c=1.1505466726$, $k=0.820125$ and $\varepsilon=0.5$, initial condition $\left(x_{1},x_{2},y_{1},y_{2}\right)=\left(2,2,0,0\right)$. These numerics illustrate the combination of large and small oscillations in a persistent and irregular way, possibly ``chaotic'' in a sense that deserves to be explored.
            (A) - Time course for $x_1(t)$. 
            (B) - Projection of the trajectory (blue) on the $\left(x_1,y_1\right)$ plane, initial condition on the blue dot, red critical manifold $C_0$. 
            (C) - Time course for $y_1(t)$.
            (D) - Projection of the trajectory (blue) on the $\left(x_1,x_2\right)$ plane, initial condition on the blue dot.
            (E) - Projection of the trajectory (blue) on the $\left(x_2,y_2\right)$ plane, initial condition on the blue dot, red critical manifold $C_0$, equilibrium on the white dot. 
            (F) - Time course for $x_2(t)$ (blue) and $y_2(t)$ (orange). 
         }
   \label{fig:chaoticMMO}
\end{figure}

\section*{Acknowledgments}
The first and second authors had financial support from CMUP, member of LASI, (UIDP/00144/2020),  financed by Funda\ca o para a Ci\^encia e a Tecnologia, Portugal (FCT/MCTES) through national funds. 
The third author has been supported by the Project CEMAPRE/REM (UIDB/05069/2020) financed by FCT/MCTES through national funds. The authors are also indebted to two reviewers for the suggestions which helped to improve the readability of this manuscript.


\end{document}